\documentclass[reqno]{amsart}
\usepackage{amscd}
\usepackage{amssymb}
\usepackage[dvips]{graphics}
\usepackage[margin=1in,includeheadfoot]{geometry}
\usepackage{url}
\usepackage{hyperref}
\usepackage{cite}

\def\beq{\begin{equation}}
\def\eeq{\end{equation}}
\def\ba{\begin{array}}
\def\ea{\end{array}}

\def\R{\mathbb R}

\newtheorem{thm}{Theorem}[section]
\newtheorem{lm}[thm]{Lemma}

\theoremstyle{definition}
\newtheorem{rem}[thm]{Remark}

\newtheorem{df}[thm]{Definition}

\newtheorem{assump}{Assumption}

\theoremstyle{remark}

\begin{document}
\pagestyle{plain}\today
\title{Boundary Lipschitz regularity of solutions for general semilinear elliptic equations in divergence form}

\author{Jingqi Liang\\
 \small{School of Mathematical Sciences, CMA-Shanghai, Shanghai Jiao Tong University}\\
 \small{Shanghai, China}\\\\
\small{ Lihe Wang}\\
 \small{Department of Mathematics, University of Iowa, Iowa City, IA, USA; School of Mathematical Sciences, Shanghai Jiao Tong University}\\
 \small{Shanghai, China}\\\\
\small{Chunqin Zhou}\\
 \small{School of Mathematical Sciences, CMA-Shanghai, Shanghai Jiao Tong University}\\
 \small{Shanghai, China}}

\footnote{The third author was partially supported by NSFC Grant 11771285 and 12031012.}

\begin{abstract}
In this paper, we study the nonhomogeneous Dirichlet problem concerning general semilinear elliptic equations in divergence form. We establish that the boundary Lipschitz regularity of solutions under some more weaker conditions on the coefficients, the boundary, the boundary function and the nonhomogeneous term. In particular, we assume that the nonhomogeneous term satisfies Dini continuity condition and Lipschitz Newtonian potential condition, which will be the optimal conditions to obtain  the boundary Lipschitz regularity of solutions.
\end{abstract}

\maketitle

{\bf Keywords:} Boundary Lipschitz regularity, Semilinear elliptic equation, Dini condition.

\section{Introduction}
In this work, we study the boundary Lipschitz regularity of solutions to the following semilinear elliptic equations in divergence form:
\begin{equation}\label{eq1}
\left\{
\begin{array}{rcll}
-D_{j}(a_{ij}D_{i}u)+b_{i}D_{i}u&=&-\text {div}\overrightarrow{\mathbf{F}}(x,u)\qquad&\text{in}~~\Omega,\\
u&=&g\qquad&\text{on}~~\partial\Omega, \\
\end{array}
\right.
\end{equation}
where $\Omega$ is a bounded domain in $\mathbb{R}^{n}$. The coefficients $(a_{ij})_{n\times n},~b_{i}$ are assumed to be measurable functions on $\Omega$, $(a_{ij}(x))_{n\times n}$ is symmetric and there are constants $\Lambda_{1},~\Lambda_{2}>0$ so that
\begin{equation}\label{aij}
\Lambda_{1}|\xi|^{2}\leq \sum\limits_{ij}a_{ij}(x)\xi_{i}\xi_{j}\leq\Lambda_{2}|\xi|^{2},~~\forall x\in \Omega,~~\xi\in \mathbb{R}^{n},
\end{equation}
and for some $q>n$
  \begin{equation}\label{b}
  \sum_{i}\|b_{i}\|_{L^{q}(\Omega)}\leq\Lambda_{2}.
  \end{equation}

An interesting question is what are the optimal conditions on the nonhomogeneous term $\overrightarrow{\mathbf{F}}(x,u)$ together with suitable conditions on the leading coefficients $a_{ij}$,   the boundary $\partial \Omega$  and the boundary function $g$ to obtain the pointwise boundary $C^{1}$ or Lipschitz regularity of solutions?

For the classical Poisson equation
$$\Delta u=f,
$$
in 2013, Andersson, Lindgren and Shahgholian  \cite {ALS} showed that the weakest assumption to get the $C^{1,1}$ regularity of $u$ is $f*N$ is $C^{1,1}$, where $N$ is the Newtonian potential and $*$ denotes the convolution. This condition on $f$ usually is called as $C^{1,1}$ Newtonian potential condition.

For the semilinear elliptic equation
$$\Delta u=f(x,u)\quad \text{in}~~~B_{1},$$
Shahgholian showed the interior $C^{1,1}$ regularity of $u$ under the conditions that $f(x,u)$ is Lipschitz continuous in $x$, uniformly in $u$, and $\partial_{u}f\geq-C$. The sharp conditions on $f(x,u)$ which ensure the $C^{1,1}$ regularity of $u$ were given by Indrei, Minne and Nurbekyan \cite{1} in 2017. They showed that the interior $C^{1,1}$ regularity of $u$ when $f(x,u)$ satisfies the uniform Dini continuity condition in $u$ and the uniform $C^{1,1}$ Newtonian potential condition in $x$.

For the linear elliptic equations in divergence form, Burch \cite {Bu} in 1978 showed that the weak solutions to
$$
-D_{j}(a_{ij}D_{i}u)+D_{i}(b_{i}u)+cu=f\quad\text{in}~~\Omega,
$$
are $C^2(\Omega)$ under the main assumptions of all partial derivatives of order 1 of $a_{ij}$ and $b_i$ together with $f$ being locally Dini continuous over $\Omega$.

For the semilinear elliptic equations in divergence form, in our latest paper \cite{17}, we studied the Dirichlet problem
\begin{equation}\label{eq2}
\left\{
\begin{array}{rcll}
\Delta u&=&\text {div}\overrightarrow{\mathbf{F}}(x,u)\qquad&\text{in}~~\Omega,\\
u&=&g\qquad&\text{on}~~\partial\Omega, \\
\end{array}
\right.
\end{equation}
and provide some sharp conditions on nonhomogeneous term and the boundary to ensure  the pointwise boundary Lipschitz regularity of $u$. Precisely, if the domain satisfies $C^{1,\text{Dini}}$ condition or Reifenberg $C^{1,\text{Dini}}$ condition at a boundary point $x_{0}$, and $\overrightarrow{\mathbf{F}}(x,u)$ satisfies uniform Dini continuity condition in $u$ and uniform Lipschitz Newtonian potential condition in $x$, then the solution is Lipschitz continuous at $x_{0}$.

Along this line of consideration, in this paper, we would like to generalize the results to more general semilinear elliptic equations $(\ref{eq1})$. Besides some suitable conditions on   $\overrightarrow{\mathbf{F}}(x,u)$, what weaker conditions could be proposed for $a_{ij}$?  In this purpose, let us recall the uniform elliptic equation in non-divergence form only with leading terms, which was considered in \cite {Ca},
\begin{equation}\label{saij}
a_{ij}(x)D_{ij}u=f.
\end{equation}
Caffarelli in \cite{Ca} comprehensively studied the qualitative properties of the viscosity solutions of $(\ref{saij})$ such as $L^{\infty}$ estimates, Harnack inequality, H$\ddot{o}$lder regularity theory and Calderon Zygmand $L^{p}$ estimates. In particular, if the coefficients $a_{ij}$ are uniformly close to being constant and $f$ has controlled growth, then a bounded solution $u$ in $B_1$ must be a weak solution of class $C^{1,\alpha}$, with an $\alpha$ that approaches 1 when $a_{ij}$  become closer to constants. Furthermore, if $a_{ij}$ are $C^{\alpha}$ and $f$ is $C^{\alpha}$, then $u$ belongs to $C^{2,\alpha}$ in the interior of $B_{1}$.

Motivated by the above results, we would like to provide the proper Dini decay conditions of the $L^{\infty}$ norm of $|a_{ij}-\delta_{ij}|$. Before we state our main results, we give some notations and definition.

{\bf Notations:}

$(\delta_{ij})_{n\times n}$:  the $n-$order identity matrix.

$|x|:=\sqrt{\sum\limits_{i=1}^{n} x_{i}^{2}}$: the Euclidean norm of $x=\left(x_{1}, x_{2}, \ldots, x_{n}\right) \in \mathbb{R}^{n}$.

$|x^{'}|:=\sqrt{\sum\limits_{i=1}^{n-1} x_{i}^{2}}$: the Euclidean norm of $x^{'}=\left(x_{1}, x_{2}, \ldots, x_{n-1}\right) \in \mathbb{R}^{n-1}$.

$B_{r}(x_{0}):=\left\{x \in \mathbb{R}^{n}:|x-x_{0}|<r\right\}$.

$B_{r}:=\left\{x \in \mathbb{R}^{n}:|x|<r\right\}$.

$\Omega_{r}:=B_{r}\cap\Omega$.

$T_{r}:=B_{r}\cap\{x_{n}=0\}=\left\{(x^{'},0) \in \mathbb{R}^{n}:|x^{'}|<r\right\}$.

$\vec{a}\cdot\vec{b}$: the standard inner product of $\vec{a}, \vec{b} \in \mathbb{R}^{n} .$

$\{\vec{e}_{i}\}_{i=1}^{n}$: the standard basis of $\mathbb{R}^{n}$.

\begin{df}\label{boundarydef} Let $x_{0}\in\partial\Omega$.   We say that $\partial\Omega$ is $C^{1,\text{Dini}}$ at $x_{0},$ if there exists a unit vector $\vec{n}$ and a positive constant $r_{0}$, a Dini modulus of continuity $\omega(r)$ satisfying
$$\int_{0}^{r_{0}} \frac{\omega(r)}{r} d r<\infty$$
such that for any $0<r\leq r_{0}$,
$$
B_{r}(x_0) \cap\left\{x\in\mathbb{R}^{n}:(x-x_{0}) \cdot \vec{n}>r \omega(r)\right\}\subset B_{r}(x_{0}) \cap \Omega \subset B_{r}(x_0) \cap\left\{x\in\mathbb{R}^{n}:(x-x_{0}) \cdot \vec{n}>-r \omega(r)\right\}.
$$

We say $\partial\Omega$ is $C^{1,\text{Dini}}$ if for any $x_{0}\in\partial\Omega$, $\partial\Omega$ is $C^{1,\text{Dini}}$ at $x_{0}$ with Dini modulus of continuity $\omega(r)$.
\end{df}
\begin{rem}
Any modulus of continuity $\omega(t)$ is non-decreasing, subadditive, continuous and satisfies $\omega(0)=0$ (see \cite{2}). Hence any modulus of continuity $\omega(t)$ satisfies
\begin{equation}\label{diniproperty}\frac{\omega(r)}{r}\leq2\frac{\omega(h)}{h},\quad 0<h<r.
\end{equation}
\end{rem}
\begin{df}\label{exrei}(Reifenberg $C^{1,\text{Dini}}$ condition). Let $x_{0}\in\partial\Omega.$ We say that $\Omega$ satisfies the $(r_{0}, \omega)$-Reifenberg $C^{1,\text{Dini}}$ condition at $x_{0}$ if there exists a positive constant $r_{0}$ and a Dini modulus of continuity $\omega(r)$ satisfying $\int_{0}^{r_{0}} \frac{\omega(r)}{r} d r<\infty$ such that for any $0<r\leq r_{0},$ there exists a unit vector $\vec{n}_{r}\in \mathbb{R}^{n}$ such that
$$
B_{r}(x_0) \cap\left\{x\in\mathbb{R}^{n}:(x-x_{0}) \cdot \vec{n}_{r}>r \omega(r)\right\}\subset B_{r}(x_{0}) \cap \Omega \subset B_{r}(x_0) \cap\left\{x\in\mathbb{R}^{n}:(x-x_{0}) \cdot \vec{n}_{r}>-r \omega(r)\right\}.
$$

We say $\partial\Omega$ is Reifenberg $C^{1,\text{Dini}}$ if for any $x_{0}\in\partial\Omega$, $\Omega$ satisfies Reifenberg $C^{1,\text{Dini}}$ condition at $x_{0}$ with Dini modulus of continuity $\omega(r)$.
\end{df}
The following lemma can be found in \cite{5}.
\begin{lm}\label{rem1.4}
If $\Omega$ satisfies $(r_{0}, \omega)$-Reifenberg $C^{1,\text{Dini}}$ condition, then there exists a bounded nonnegative function $S(\theta)\geq1$ such that $\left|\vec{n}_{r}-\vec{n}_{\theta r}\right| \leq S(\theta) \omega(r)$ for each $0<\theta<1$ and $0<r\leq r_{0}.$ Furthermore, for a fixed positive constant $0<\lambda<1,$ $\{\vec{n}_{\lambda^{i}r_{0}}\}_{i=0}^{\infty}$ is a Cauchy sequence. We can set $\lim\limits_{i\rightarrow\infty}\vec{n}_{\lambda^{i}r_{0}}=\vec{n}_{*}$.
\end{lm}
\begin{rem}\label{rem1.5}
In the sequel, we assume $\partial\Omega$ is $C^{1,\text{Dini}}$ or Reifenberg $C^{1,\text{Dini}}$. Then by the definition, for each $x_{0}\in\partial\Omega$ and $\omega(r_{0})<1$, it follows that
$$\lim_{r\rightarrow0}\frac{|B_{r}(x_{0})\cap\Omega^{c}|}{|B_{r}(x_{0})|}=\frac{1}{2}.
$$
\end{rem}
\begin{df}\label{bvdini}
Let $x_{0} \in \partial \Omega $.  The boundary value $g$ is said to be $C^{1,\text{Dini}}$ at $x_{0}$ with respect to a function $v_{x_0}(x)$,   if there exists a constant vector $\vec{a}$, a positive constant $r_{0}$ and a Dini modulus of continuity $\sigma(r)$ satisfying $\int_{0}^{r_{0}} \frac{\sigma(r)}{r} d r<\infty$ such that for any $0<r \leq r_{0}$ and $x \in \partial \Omega \cap B_{r}(x_{0})$,
$$
|g(x)-v_{x_0}(x)-(g(x_{0})-v_{x_0}(x_{0}))-\vec{a} \cdot(x-x_{0})| \leq r \sigma(r).
$$
\end{df}

Next we propose the following assumptions on $\overrightarrow{\mathbf{F}}(x,u)$ and all $a_{ij}$.
\begin{assump}\label{as1} $\overrightarrow{\mathbf{F}}(x,t)\in L^{\infty}(B_{d}\times \R)$ where $d$ is large enough. Moreover
$\overrightarrow{\mathbf{F}}(x,t)$ is Dini continuous in $t$ with continuity modulus $\varphi(r)$, uniformly in $x$, i.e.
$$\left|\overrightarrow{\mathbf{F}}(x,t_{2})-\overrightarrow{\mathbf{F}}(x,t_{1})\right|\leq\varphi(|t_{2}-t_{1}|),
$$
and $\int_{0}^{t_{0}} \frac{\varphi(t)}{t} dt<\infty$, for some $t_{0}>0.$
\end{assump}
\begin{assump}\label{as2}
For every boundary point $x_{0}$ and each $t\in \R$, there exists a function $v_{x_{0}}(\cdot,t)$ in $B_{1}(x_{0})$ satisfying
$$
\Delta v_{x_{0}}(\cdot, t)=\text{div}\overrightarrow{\mathbf{F}}(\cdot,t) \quad \text{in}~~B_{1}(x_{0}).
$$
Furthermore, $v_{x_{0}}(\cdot,t)$  is a Lipschitz function which is  uniform in $x_0$ and $t$ with Lipschitz constant $T$.
\end{assump}
\begin{rem}
(1) We can always assume that $\Omega\subset B_{\frac{d}{2}}$.\\
(2) In the sequel, for $g(x)\in L^{\infty}(\partial\Omega)$ is the boundary value of $(\ref{eq1})$, then for every boundary point $x_{0}$, we let $v_{x_{0}}(x)$ solve
$$\Delta v_{x_{0}}(x)=\text{div}\overrightarrow{\mathbf{F}}(x,g(x_{0})) \quad \text{in}~~B_{1}(x_{0}).$$
\end{rem}
We also assume the following conditions on $a_{ij}$.

\begin{assump}\label{as3} For each $x_{0}\in \partial \Omega$, there exists a Dini modulus of continuity $\omega_{1}(r)$ satisfying $\int_{0}^{r_{0}} \frac{\omega_{1}(r)}{r} d r<\infty$ and a positive constant $r_{0}$ such that for any $0<r\leq r_{0}$,
    $$\|a_{ij}(x)-\delta_{ij}\|_{L^{\infty}(B_{r}(x_{0})\cap\Omega)}\leq\omega_{1}(r).$$
 \end{assump}

\begin{thm}\label{mr1}
Let $x_{0} \in \partial \Omega .$ $\overrightarrow{\mathbf{F}}(x,t)$ and all $a_{ij}$ satisfy Assumption \ref{as1}, \ref{as2} and \ref{as3}. If $\partial\Omega$ is $C^{1,\text{Dini}}$ at $x_{0}$ and $g$ is $C^{1,\text{Dini}}$ at $x_{0}$ with respect to $v_{x_{0}}(x)$, then the solution of $(\ref{eq1})$, $(\ref{aij})$ and $(\ref{b})$ is Lipschitz continuous at $x_{0}$, i.e.
$$|u(x)-u(x_{0})|\leq C|x-x_{0}|,\quad \forall~x\in B_{1}\cap\Omega.$$
where $C=C(n, \Lambda_{1}, \Lambda_{2}, \Omega, T, \|u\|_{L^{\infty}(\Omega)}, \|g\|_{L^{\infty}(\partial\Omega)})$.
\end{thm}
\begin{thm}\label{mr2}
Let $x_{0} \in \partial \Omega .$ $\overrightarrow{\mathbf{F}}(x,t)$ and all $a_{ij}$ satisfy Assumption \ref{as1}, \ref{as2} and \ref{as3}. If $\partial\Omega$ satisfies Reifenberg $C^{1,\text{Dini}}$ condition at $x_{0}$ and $g$ is $C^{1,\text{Dini}}$ at $x_{0}$ with respect to $v_{x_{0}}(x)$,
then the solution of $(\ref{eq1})$, $(\ref{aij})$ and $(\ref{b})$ is Lipschitz continuous at $x_{0}$.
\end{thm}

As for the pointwise boundary $C^{1}$ or Lipschitz regularity of solutions to equations in non-divergence form or the fully nonlinear equations, there have been extensive results in the past two decades.  For the following second order uniformly elliptic equations,
\begin{eqnarray*}
\left\{
\begin{array}{rcll}
-a_{ij}(x)\frac{\partial^{2}u}{\partial x_{i}\partial x_{j}}&=&f(x) \qquad&\text{in}~~\Omega,\\
u&=&0\qquad&\text{on}~~\partial\Omega, \\
\end{array}
\right.
\end{eqnarray*}
Li and Wang proved that the solution is differentiable at any boundary point when the domain is convex and they extended their results to the nonhomogeneous boundary value Dirichlet problem (see \cite{13}, \cite{14}). In \cite{4} and \cite{5}, Huang, Li and Wang got the boundary Lipschitz regularity under more general geometrical conditions, $C^{1,\text{Dini}}$ condition or Reifenberg $C^{1,\text{Dini}}$ condition. Furthermore if $\partial\Omega$ is punctually $C^{1}$ additionally, they obtained the boundary differentiability regularity. They used an iteration method and their main tools are the Alexandroff-Bakelman-Pucci maximum principle, Harnack inequality and barrier technique.

Besides, there are also many important results in the fully nonlinear elliptic equations. In \cite{3}, Ma and Wang proved the boundary differentiability of viscosity solutions for fully nonlinear elliptic equations under Dini conditions and boundary $C^{1,\alpha}$ regularity was also obtained as a corollary. Recently, Lian, Wu and Zhang \cite{17} considered the following equations,
\begin{eqnarray*}
\left\{
\begin{array}{rcll}
u&\in&S(\lambda,~\Lambda,~f) \qquad&\text{in}~~\Omega,\\
u&=&g\qquad&\text{on}~~\partial\Omega,\\
\end{array}
\right.
\end{eqnarray*}
where $S(\lambda,~\Lambda,~f)$ denotes the Pucci class with uniform constants $\lambda$ and $\Lambda$. They proved that $u$ is Lipschitz continuous at $x_{0}\in\partial\Omega$ if the domain $\Omega$ satisfies exterior Reifenberg $C^{1,\text{Dini}}$ condition at $x_{0}$ and the boundary value $g$ is $C^{1,\text{Dini}}$ at $x_{0}$.

For elliptic equations in the divergence form, to prove Lipschitz regularity of solutions, the basic idea is to seek for suitable harmonic functions to be compared solutions with, which is mainly different from the technique used to treat equations in the non-divergence form and the fully nonlinear elliptic equations. We will show our key idea to prove Theorem $\ref{mr1}$  in section 2. We first approximate $u$ by a harmonic function $h$ in $B_{\frac{1}{16}}\cap\Omega$ and then approximate $h$ by a linear polynomial in a sufficiently small ball. In section 3 we prove the boundary Lipschitz regularity by an iteration method. A scaling argument allows us to iterate the approximation result in Section 2 and then construct a sequence of linear polynomials converging uniformly to a desired one.
\section{Preliminary tools}
In this section firstly we will give a general approximation lemma of the following elliptic equation:
\begin{eqnarray*}
-D_{j}(a_{ij}D_{i}u)+b_{i}D_{i}u=f-\text{div}\overrightarrow{\mathbf{F}} \quad \text{in}~~\Omega,
\end{eqnarray*}
where $\Omega$ is a bounded domain. Our aim is to approximate $u$ by a linear function. To begin with, we find a proper harmonic function $h$ to approximate $u$, then the linear function is derived from $h$ naturally. To prove this, our main tool is the following local maximum principle and $C^{\alpha}$ estimates, Theorem 8.16, Theorem 8.24 and Theorem 8.29 in \cite{DT}.
\begin{df}
We say $u\in W^{1,2}(\Omega)$ is a weak solution of
\begin{equation}\label{2}
-D_{j}(a_{ij}D_{i}u)+b_{i}D_{i}u=f-\text{div}\overrightarrow{\mathbf{F}}\qquad\text{in}~~\Omega,
\end{equation}
if for all $\phi\in W^{1,2}_{0}(\Omega)$, which is called a test function, we have
$$\int_{\Omega}a_{ij}D_{i}uD_{j}\phi+b_{i}D_{i}u\phi dx=\int_{\Omega}f\phi+\overrightarrow{\mathbf{F}}\cdot D\phi dx.
$$
\end{df}
\begin{thm}\label{lemma1}
Let $\Omega$ be a bounded domain in $\mathbb{R}^{n}.$ We suppose that $\overrightarrow{\mathbf{F}}\in L^{p}(\Omega),$ $f\in L^{\frac{p}{2}}(\Omega),$ for some $p>n$, $a_{ij}$ and $b_{i}$ satisfy $(\ref{aij})$ and $(\ref{b})$. Then if $u$ is a $W^{1,2}(\Omega)$ solution of $(\ref{2})$, then we have
$$\|u\|_{L^{\infty}(\Omega)}\leq\|u\|_{L^{\infty}(\partial\Omega)}+C(\|\overrightarrow{\mathbf{F}}\|_{L^{p}(\Omega)}+\|f\|_{L^{\frac{p}{2}}(\Omega)}),
$$
where $C=C(n,\Lambda_{1}, \Lambda_{2}, p, |\Omega|).$
\end{thm}
\begin{thm}[Interior H$\ddot{o}$lder estimates]\label{lin}
Let $a_{ij}, b_{i}$ satisfy conditions $(\ref{aij})$ and $(\ref{b})$, let $\overrightarrow{\mathbf{F}}\in L^{p}(\Omega)$, $f\in L^{\frac{p}{2}}(\Omega)$ for some $p>n$. Then if $u\in W^{1,2}(\Omega)$ is a solution of $(\ref{2})$, we have for any $\Omega'\Subset\Omega$ the estimate
$$\|u\|_{C^{\alpha}(\overline{\Omega'})}\leq C(\|u\|_{L^{\infty}(\Omega)}+\|\overrightarrow{\mathbf{F}}\|_{L^{p}(\Omega)}
+\|f\|_{L^{\frac{p}{2}}(\Omega)}),
$$
where $C=C(n,\Lambda_{1},\Lambda_{2},\Omega,p,d')$, $d'=dist(\Omega', \partial\Omega)$, $\alpha=\alpha(n,\Lambda_{1},\Lambda_{2},\Omega,p,d')$.
\end{thm}
\begin{thm}[Global H$\ddot{o}$lder estimates]\label{lemma2} Let $\Omega$ is a bounded domain in $\mathbb{R}^{n}$ and $T$ is a boundary portion. Assume that there exists $c_{0}>0$ such that for arbitrary $x_{0}\in T$,
$$\liminf_{r\rightarrow0}\frac{|B_{r}(x_{0})\cap\Omega^{c}|}{|B_{r}(x_{0})|}=c_{0}.
$$
Let $a_{ij}, b_{i}$ satisfy conditions $(\ref{aij})$ and $(\ref{b})$, let $\overrightarrow{\mathbf{F}}\in L^{p}(\Omega)$, $f\in L^{\frac{p}{2}}(\Omega)$ for some $p>n$. Then if $u\in W^{1,2}(\Omega)$ satisfies equation $(\ref{2})$ in $\Omega$ and there exist constants $K,\alpha_{0}>0$ such that
$$\underset{B_{R}(x_{0})\cap\partial\Omega}{osc}u\leq KR^{\alpha_{0}},\quad \forall x_{0}\in T,~R>0,
$$
it follows that $u\in C^{\alpha}(\Omega\cup T)$ for some $\alpha>0$ and for any $\Omega'\Subset\Omega \cup T$,
\begin{eqnarray*}
\|u\|_{C^{\alpha}(\Omega')}\leq C(\|u\|_{L^{\infty}(\Omega)}+\|\overrightarrow{\mathbf{F}}\|_{L^{p}(\Omega)}
+\|f\|_{L^{\frac{p}{2}}(\Omega)}+K),
\end{eqnarray*}
where $\alpha=\alpha(n,\Lambda_{1},\Lambda_{2},\Omega,p,\alpha_{0},d')$, $C=C(n,\Lambda_{1},\Lambda_{2},\Omega,p,\alpha_{0},d')$, $d'=dist(\Omega', \partial\Omega-T)$.
\end{thm}

Then we give an approximation lemma.
\begin{lm}\label{fs}
Let $\Omega$ be a bounded domain satisfying that there exists $c_{0}>0$ and such that for arbitrary $x_{0}\in B_{1}\cap\partial\Omega$,
$$\liminf_{r\rightarrow0}\frac{|B_{r}(x_{0})\cap\Omega^{c}|}{|B_{r}(x_{0})|}=c_{0}.
$$
 Assume that $0\in\partial\Omega$  and $u$ satisfies weakly
\begin{eqnarray*}
\left\{
\begin{array}{rcll}
-D_{j}(a_{ij}D_{i}u)+b_{i}D_{i}u&=&f-\text{div}\overrightarrow{\mathbf{F}}\qquad&\text{in}~~B_{1}\cap\Omega,\\
u&=&g\qquad&\text{on}~~B_{1}\cap\partial\Omega, \\
\end{array}
\right.
\end{eqnarray*}
with $\overrightarrow{\mathbf{F}}(x)\in L^{\infty}(\Omega),$ $f\in L^{q}(\Omega)$, $g\in L^{\infty}(\partial\Omega)$. If $a_{ij}$ and $b_{i}$ satisfy $(\ref{aij}),~(\ref{b})$, and
$$\|a_{ij}-\delta_{ij}\|_{L^{\infty}(B_{1}\cap\Omega)}\leq \varepsilon_{1}, \quad \|b_{i}\|_{L^{q}(B_{1}\cap\Omega)}\leq \varepsilon_{2},
$$
for some $\varepsilon_{1},~\varepsilon_{2}>0$ small enough, and
$$B_{1}\cap\{x\in\mathbb{R}^{n}: x_{n}>\varepsilon\}\subset B_{1}\cap\Omega\subset B_{1}\cap\{x\in\mathbb{R}^{n}: x_{n}>-\varepsilon\}$$
for some $0<\varepsilon<\frac{1}{16}$, then, for some $\alpha>0$, there exists a universal constant $C_{0}$ and a harmonic function $h$ defined in $B_{\frac{1}{16}}$ which is odd with respect to $x_{n}$ satisfying
$$\|h\|_{L^{\infty}(B_{\frac{1}{16}})}\leq(1+2C_{0}\varepsilon)\|u\|_{L^{\infty}(B_{1}\cap\Omega)}$$
such that
\begin{eqnarray*}
\left\|u-h\right\|_{L^{\infty}(B_{\frac{1}{16}}\cap\Omega)}
&\leq&C\|g\|_{L^{\infty}(B_{1}\cap\partial\Omega)}+C(\|\overrightarrow{\mathbf{F}}\|_{L^{2q}(B_{1}\cap\Omega)}+\|f\|_{L^{q}(B_{1}\cap\Omega)})\\
&&+C(\varepsilon^{\alpha}+(\varepsilon_{1}+\varepsilon_{2})^{\frac{\alpha}{6}}+\sqrt{\varepsilon_{1}+\varepsilon_{2}})\|u\|_{L^{\infty}(B_{1}\cap\Omega)}.
\end{eqnarray*}
where $\alpha=\alpha(n,q,\Lambda_{1},\Lambda_{2},\Omega)>0$, $C=C(n,q,\Lambda_{1},\Lambda_{2},\Omega)$.
\end{lm}

\begin{proof}
We divide the proof into Four steps.

\
\

{\bf Step 1:} Prove $
|u(x)|\leq C^{*}(x_{n}+\varepsilon)^{\alpha}+3\|g\|_{L^{\infty}(B_{1}\cap\partial\Omega)}~ \text{in}~~\overline{B_{\frac{1}{2}}\cap\Omega}
$ for some $C^*$.

\
\

We first consider the interior H$\ddot{o}$lder estimates of $u$. Since $B_{1}\cap\{x\in\mathbb{R}^{n}: x_{n}>\varepsilon\}\subset B_{1}\cap\Omega\subset B_{1}\cap\{x\in\mathbb{R}^{n}: x_{n}>-\varepsilon\}$, then $B_{\frac{1}{2}}\cap\{x\in\mathbb{R}^{n}: x_{n}>\varepsilon\}\Subset B_{1}\cap\Omega$. Then by Theorem $\ref{lin}$, there exists $0<\beta<1$ such that
\begin{equation}\label{ii}
\|u\|_{C^{\beta}(\overline{B_{\frac{1}{2}}\cap\{x_{n}>\varepsilon\}})}\leq C(\|u\|_{L^{\infty}(B_{1}\cap\Omega)}+\|\overrightarrow{\mathbf{F}}\|_{L^{2q}(B_{1}\cap\Omega)}+\|f\|_{L^{q}(B_{1}\cap\Omega)}).
\end{equation}

Next we consider the H$\ddot{o}$lder estimates of $u$ up to the boundary. Let $u_{1}$ and $u_{2}$ solve the following equations.
\begin{eqnarray*}
\left\{
\begin{array}{rcll}
-D_{j}(a_{ij}D_{i}u_{1})+b_{i}D_{i}u_{1}&=&f-\text{div}\overrightarrow{\mathbf{F}}\qquad&\text{in}~~B_{1}\cap\Omega,\\
u_{1}&=&\sup\limits_{B_{1}\cap\partial\Omega}g\qquad&\text{on}~~B_{1}\cap\partial\Omega, \\
u_{1}&=&u\qquad&\text{on}~~\partial B_{1}\cap\Omega,
\end{array}
\right.
\end{eqnarray*}
\begin{eqnarray*}
\left\{
\begin{array}{rcll}
-D_{j}(a_{ij}D_{i}u_{2})+b_{i}D_{i}u_{2}&=&f-\text{div}\overrightarrow{\mathbf{F}}\qquad&\text{in}~~B_{1}\cap\Omega,\\
u_{2}&=&\inf\limits_{B_{1}\cap\partial\Omega}g\qquad&\text{on}~~B_{1}\cap\partial\Omega, \\
u_{2}&=&u\qquad&\text{on}~~\partial B_{1}\cap\Omega.
\end{array}
\right.
\end{eqnarray*}
Then $u_{1}-u$ and $u_{2}-u$ satisfy
\begin{eqnarray*}
\left\{
\begin{array}{rcll}
-D_{j}(a_{ij}D_{i}(u_{1}-u))+b_{i}D_{i}(u_{1}-u)&=&0\qquad&\text{in}~~B_{1}\cap\Omega,\\
u_{1}-u&\geq&0\qquad&\text{on}~~B_{1}\cap\partial\Omega, \\
u_{1}-u&=&0\qquad&\text{on}~~\partial B_{1}\cap\Omega,
\end{array}
\right.
\end{eqnarray*}
\begin{eqnarray*}
\left\{
\begin{array}{rcll}
-D_{j}(a_{ij}D_{i}(u_{2}-u))+b_{i}D_{i}(u_{2}-u)&=&0\qquad&\text{in}~~B_{1}\cap\Omega,\\
u_{2}-u&\leq&0\qquad&\text{on}~~B_{1}\cap\partial\Omega, \\
u_{2}-u&=&0\qquad&\text{on}~~\partial B_{1}\cap\Omega.
\end{array}
\right.
\end{eqnarray*}
By maximum principle, it follows that
$$u_{2}\leq u\leq u_{1}\quad \text{in}~~B_{1}\cap\Omega.
$$
Under the assumptions of Lemma $\ref{fs}$, using global H$\ddot{o}$lder estimates of Theorem $\ref{lemma2}$ to $u_{1}$ and $u_{2}$, we can find $0<\alpha_{1},\alpha_{2}<1$ such that
\begin{eqnarray*}
\|u_{1}\|_{C^{\alpha_{1}}(\overline{B_{\frac{1}{2}}\cap\Omega})}\leq C(\|u_{1}\|_{L^{\infty}(B_{1}\cap\Omega)}+\|\overrightarrow{\mathbf{F}}\|_{L^{2q}(B_{1}\cap\Omega)}+\|f\|_{L^{q}(B_{1}\cap\Omega)}),
\end{eqnarray*}
\begin{eqnarray*}
\|u_{2}\|_{C^{\alpha_{2}}(\overline{B_{\frac{1}{2}}\cap\Omega})}\leq C(\|u_{2}\|_{L^{\infty}(B_{1}\cap\Omega)}+\|\overrightarrow{\mathbf{F}}\|_{L^{2q}(B_{1}\cap\Omega)}+\|f\|_{L^{q}(B_{1}\cap\Omega)}).
\end{eqnarray*}
By Theorem $\ref{lemma1}$, we can estimate $\|u_{1}\|_{L^{\infty}(B_{1}\cap\Omega)}$ and $\|u_{2}\|_{L^{\infty}(B_{1}\cap\Omega)}$ to get
$$\|u_{1}\|_{L^{\infty}(B_{1}\cap\Omega)}\leq C(\|u\|_{L^{\infty}(B_{1}\cap\Omega)}+\|g\|_{L^{\infty}(B_{1}\cap\partial\Omega)}
+\|\overrightarrow{\mathbf{F}}\|_{L^{2q}(B_{1}\cap\Omega)}+\|f\|_{L^{q}(B_{1}\cap\Omega)}),
$$
$$\|u_{2}\|_{L^{\infty}(B_{1}\cap\Omega)}\leq C(\|u\|_{L^{\infty}(B_{1}\cap\Omega)}+\|g\|_{L^{\infty}(B_{1}\cap\partial\Omega)}
+\|\overrightarrow{\mathbf{F}}\|_{L^{2q}(B_{1}\cap\Omega)}+\|f\|_{L^{q}(B_{1}\cap\Omega)}).
$$
Putting the above estimates together, we have
\begin{equation}\label{alpha1}
\|u_{1}\|_{C^{\alpha_{1}}(\overline{B_{\frac{1}{2}}\cap\Omega})}\leq C(\|u\|_{L^{\infty}(B_{1}\cap\Omega)}+\|g\|_{L^{\infty}(B_{1}\cap\partial\Omega)}
+\|\overrightarrow{\mathbf{F}}\|_{L^{2q}(B_{1}\cap\Omega)}+\|f\|_{L^{q}(B_{1}\cap\Omega)}),
\end{equation}
\begin{equation}\label{alpha1}
\|u_{2}\|_{C^{\alpha_{2}}(\overline{B_{\frac{1}{2}}\cap\Omega})}\leq C(\|u\|_{L^{\infty}(B_{1}\cap\Omega)}+\|g\|_{L^{\infty}(B_{1}\cap\partial\Omega)}
+\|\overrightarrow{\mathbf{F}}\|_{L^{2q}(B_{1}\cap\Omega)}+\|f\|_{L^{q}(B_{1}\cap\Omega)}).
\end{equation}

We denote $C^{*}=C(\|u\|_{L^{\infty}(B_{1}\cap\Omega)}+\|\overrightarrow{\mathbf{F}}\|_{L^{2q}(B_{1}\cap\Omega)}+\|f\|_{L^{q}(B_{1}\cap\Omega)}
+\|g\|_{L^{\infty}(B_{1}\cap\partial\Omega)})$ and set $\alpha=\min\{\alpha_{1},\alpha_{2},\beta\}$.
Since $B_{1}\cap\{x\in\mathbb{R}^{n}: x_{n}>\varepsilon\}\subset B_{1}\cap\Omega\subset B_{1}\cap\{x\in\mathbb{R}^{n}: x_{n}>-\varepsilon\}$, then for any $x=(x_{0}',x_{n})\in \overline{B_{\frac{1}{2}}\cap\Omega}$, $x_{0}=(x_{0}',x_{0,n})\in B_{\frac{1}{2}}\cap\partial\Omega$, we have
$$\frac{|u_{1}(x)-u_{1}(x_{0})|}{(x_{n}+\varepsilon)^{\alpha}}=
\frac{\left|u_{1}(x)-\sup\limits_{B_{1}\cap\partial\Omega}g\right|}{(x_{n}+\varepsilon)^{\alpha}}
\leq C^{*},
$$
$$\frac{|u_{2}(x)-u_{2}(x_{0})|}{(x_{n}+\varepsilon)^{\alpha}}=
\frac{\left|u_{2}(x)-\inf\limits_{B_{1}\cap\partial\Omega}g\right|}{(x_{n}+\varepsilon)^{\alpha}}
\leq C^{*}.
$$
It follows that
\begin{eqnarray*}
\frac{u(x)-g(x_{0})}{(x_{n}+\varepsilon)^{\alpha}}&\leq&
\frac{u_{1}(x)-\inf\limits_{B_{1}\cap\partial\Omega}g}{(x_{n}+\varepsilon)^{\alpha}}\\
&\leq& \frac{\left|u_{1}(x)-\sup\limits_{B_{1}\cap\partial\Omega}g\right|}{(x_{n}+\varepsilon)^{\alpha}}+
\frac{\left|\sup\limits_{B_{1}\cap\partial\Omega}g-\inf\limits_{B_{1}\cap\partial\Omega}g\right|}{(x_{n}+\varepsilon)^{\alpha}}\\
&\leq& C^{*}+\frac{2\|g\|_{L^{\infty}(B_{1}\cap\partial\Omega)}}{(x_{n}+\varepsilon)^{\alpha}}.
\end{eqnarray*}
Then we have
$$u(x)\leq C^{*}(x_{n}+\varepsilon)^{\alpha}+3\|g\|_{L^{\infty}(B_{1}\cap\partial\Omega)}.
$$
Similarly, we can get
$$u(x)\geq -C^{*}(x_{n}+\varepsilon)^{\alpha}-3\|g\|_{L^{\infty}(B_{1}\cap\partial\Omega)}.
$$
Finally
\begin{equation}\label{uep}
|u(x)|\leq C^{*}(x_{n}+\varepsilon)^{\alpha}+3\|g\|_{L^{\infty}(B_{1}\cap\partial\Omega)}\quad \text{in}~~\overline{B_{\frac{1}{2}}\cap\Omega}.
\end{equation}

\
\

{\bf Step 2:}  Construct a harmonic function $w$ in $ B_{\frac{1}{4}}\cap\{x_{n}>-\varepsilon\}$.

\
\

We now let $w$ solve the following equation,
\begin{equation}\label{w}
\left\{
\begin{array}{rcll}
\Delta w&=&0\qquad&\text{in}~~B_{\frac{1}{4}}\cap\{x_{n}>-\varepsilon\},\\
w&=&u\qquad&\text{on}~~\partial B_{\frac{1}{4}}\cap\Omega, \\
w&=&0\qquad&\text{on}~~\partial B_{\frac{1}{4}}\cap\Omega^{c}\cap\{x_{n}>-\varepsilon\},\\
w&=&0\qquad&\text{on}~~B_{\frac{1}{4}}\cap\{x_{n}=-\varepsilon\}.
\end{array}
\right.
\end{equation}
For $w$ we have the following two results by using maximum principle and $(\ref{uep})$,
\begin{equation}\label{wmp}|w(x)|\leq \|u\|_{L^{\infty}(B_{1}\cap\Omega)}\quad\text{in}~~B_{\frac{1}{4}}\cap \{x_{n}>-\varepsilon\},
\end{equation}
$$|w(x)|=|u(x)|\leq C^{*}(x_{n}+\varepsilon)^{\alpha}+3\|g\|_{L^{\infty}(B_{1}\cap\partial\Omega)}\quad \text{on}~~\partial B_{\frac{1}{4}}\cap\Omega.
$$

By a simple calculation, we have
$$\Delta(C^{*}(x_{n}+\varepsilon)^{\alpha}+3\|g\|_{L^{\infty}(B_{1}\cap\partial\Omega)})\leq0.
$$
Then $w-(C^{*}(x_{n}+\varepsilon)^{\alpha}+3\|g\|_{L^{\infty}(B_{1}\cap\partial\Omega)})$ satisfies
\begin{eqnarray*}
\left\{
\begin{array}{rcll}
\Delta(w-(C^{*}(x_{n}+\varepsilon)^{\alpha}+3\|g\|_{L^{\infty}(B_{1}\cap\partial\Omega)}))&\geq&0\qquad&\text{in}~~B_{\frac{1}{4}}\cap\{x_{n}>-\varepsilon\},\\
w-(C^{*}(x_{n}+\varepsilon)^{\alpha}+3\|g\|_{L^{\infty}(B_{1}\cap\partial\Omega)})&\leq& 0\qquad&\text{on}~~\partial \{B_{\frac{1}{4}}\cap\{x_{n}>-\varepsilon\}\}.
\end{array}
\right.
\end{eqnarray*}
Then by maximum principle, we have
$$w(x)\leq C^{*}(x_{n}+\varepsilon)^{\alpha}+3\|g\|_{L^{\infty}(B_{1}\cap\partial\Omega)}
\quad\text{in}~~B_{\frac{1}{4}}\cap\{x_{n}\geq-\varepsilon\}.$$
By the same way, we deal with the function  $w+(C^{*}(x_{n}+\varepsilon)^{\alpha}+3\|g\|_{L^{\infty}(B_{1}\cap\partial\Omega)})$ and finally we get
\begin{equation}\label{wg}
|w(x)|\leq C^{*}(x_{n}+\varepsilon)^{\alpha}+3\|g\|_{L^{\infty}(B_{1}\cap\partial\Omega)}\quad\text{in}~~B_{\frac{1}{4}}\cap\{x_{n}\geq-\varepsilon\}.
\end{equation}
It follows that
\begin{equation}\label{w1}|w(x)|\leq C^{*}(2\varepsilon)^{\alpha}+3\|g\|_{L^{\infty}(B_{1}\cap\partial\Omega)}\quad\text{in}~~B_{\frac{1}{4}}\cap\{-\varepsilon \leq x_{n}\leq\varepsilon\}.
\end{equation}

On the other hand, for each fixed $x_{*}\in \partial B_{\frac{1}{4}}\cap\{x_{n}\geq\varepsilon+(\sqrt{\varepsilon_{1}+\varepsilon_{2}})^{\frac{1}{3}}\}$ and for $r\leq(\sqrt{\varepsilon_{1}+\varepsilon_{2}})^{\frac{1}{3}}$, we have $B_{r}(x_{*})\cap B_{\frac{1}{4}}\subset B_{\frac{1}{4}}\cap\{x_{n}\geq\varepsilon\}\subset B_{\frac{1}{2}}\cap\{x_{n}\geq\varepsilon\}$. Since $u$ is H$\ddot{o}$lder continuous in $B_{\frac{1}{2}}\cap\{x_{n}\geq\varepsilon\}$ with $\alpha=\min\{\alpha_{1},\alpha_{2},\beta\}$, if we denoting $\widehat{C}=C(\|u\|_{L^{\infty}(B_{1}\cap\Omega)}+\|\overrightarrow{\mathbf{F}}\|_{L^{2q}(B_{1}\cap\Omega)}
+\|f\|_{L^{q}(B_{1}\cap\Omega)})$, then by $(\ref{ii})$ we have
\begin{equation}\label{ubr}
|u(x)-u(x_{*})|\leq \widehat{C}r^{\beta}\leq \widehat{C}r^{\alpha},\quad x\in\overline{B_{r}(x_{*})\cap B_{\frac{1}{4}}}.
\end{equation}
Since $w=u$ on $\partial B_{\frac{1}{4}}\cap\Omega$, then $u(x_{*})=w(x_{*})$ and
\begin{equation}\label{w3}
-\widehat{C}r^{\alpha}\leq w(x)-w(x_{*})\leq \widehat{C}r^{\alpha},\quad x\in B_{r}(x_{*})\cap\partial B_{\frac{1}{4}}.
\end{equation}
In addition, by $(\ref{wmp})$ it follows that
\begin{equation}\label{w4}
-2\|u\|_{L^{\infty}(B_{1}\cap\Omega)}\leq w(x)-w(x_{*})\leq 2\|u\|_{L^{\infty}(B_{1}\cap\Omega)}\quad x\in B_{\frac{1}{4}}\cap \{x_{n}\geq\varepsilon\}.\end{equation}
Then we consider a nonnegative harmonic function $\phi(x)=\frac{x_{*}\cdot(x_{*}-x)}{2|x_{*}|r^{2}}$ in $B_{r}(x_{*})\cap B_{\frac{1}{4}}\cap\{\frac{x_{*}\cdot(x_{*}-x)}{|x_{*}|}< 2r^{2}\}$. Combining with $(\ref{w3})$ and $(\ref{w4})$ we have
\begin{eqnarray*}
\left\{
\begin{array}{rcll}
\Delta\left(w(x)-w(x_{*})-(\widehat{C}r^{\alpha}+2\|u\|_{L^{\infty}(B_{1}\cap\Omega)}\phi(x))\right)
&=&0\qquad&\text{in}~~B_{r}(x_{*})\cap B_{\frac{1}{4}}\cap\{\frac{x_{*}\cdot(x_{*}-x)}{|x_{*}|}< 2r^{2}\},\\
w(x)-w(x_{*})-(\widehat{C}r^{\alpha}+2\|u\|_{L^{\infty}(B_{1}\cap\Omega)}\phi(x))&\leq& 0\qquad&\text{on}~~B_{r}(x_{*})\cap B_{\frac{1}{4}}\cap\{\frac{x_{*}\cdot(x_{*}-x)}{|x_{*}|}= 2r^{2}\},\\
w(x)-w(x_{*})-(\widehat{C}r^{\alpha}+2\|u\|_{L^{\infty}(B_{1}\cap\Omega)}\phi(x))&\leq& 0\qquad&\text{on}~~B_{r}(x_{*})\cap \partial B_{\frac{1}{4}}.
\end{array}
\right.
\end{eqnarray*}
By maximum principle, we can get
$$w(x)-w(x_{*})\leq \widehat{C}r^{\alpha}+2\|u\|_{L^{\infty}(B_{1}\cap\Omega)}\phi(x)\quad \text{in}~~B_{r}(x_{*})\cap B_{\frac{1}{4}}\cap\{\frac{x_{*}\cdot(x_{*}-x)}{|x_{*}|}\leq 2r^{2}\}.$$
Similarly, we can also get
$$w(x)-w(x_{*})\geq -\left(\widehat{C}r^{\alpha}+2\|u\|_{L^{\infty}(B_{1}\cap\Omega)}\phi(x)\right)\quad \text{in}~~B_{r}(x_{*})\cap B_{\frac{1}{4}}\cap\{\frac{x_{*}\cdot(x_{*}-x)}{|x_{*}|}\leq 2r^{2}\}.$$
Thus,
$$|w(x)-w(x_{*})|\leq \widehat{C}r^{\alpha}+\|u\|_{L^{\infty}(B_{1}\cap\Omega)}\frac{x_{*}\cdot(x_{*}-x)}{|x_{*}|r^{2}}\quad \text{in}~~B_{r}(x_{*})\cap B_{\frac{1}{4}}\cap\{\frac{x_{*}\cdot(x_{*}-x)}{|x_{*}|}\leq 2r^{2}\}.$$
So when $x\in \overline{B_{r^{3}}(x_{*})}\cap B_{\frac{1}{4}}$ with  $r^{3}\leq\sqrt{\varepsilon_{1}+\varepsilon_{2}}$, since $\overline{B_{r^{3}}(x_{*})}\cap B_{\frac{1}{4}}\subset B_{r}(x_{*})\cap B_{\frac{1}{4}}\cap\{\frac{x_{*}\cdot(x_{*}-x)}{|x_{*}|}< 2r^{2}\}$, it follows that
\begin{eqnarray*}
|w(x)-w(x_{*})|&\leq& \widehat{C}r^{\alpha}+\|u\|_{L^{\infty}(B_{1}\cap\Omega)}r\\
&\leq&C(\|u\|_{L^{\infty}(B_{1}\cap\Omega)}+\|\overrightarrow{\mathbf{F}}\|_{L^{2q}(B_{1}\cap\Omega)}+\|f\|_{L^{q}(B_{1}\cap\Omega)}
)r^{\alpha}.
\end{eqnarray*}
Therefore for arbitrary $\hat{r}$ small enough satisfying $\hat{r}\leq \sqrt{\varepsilon_{1}+\varepsilon_{2}}$, we have
\begin{equation}\label{wf}
|w(x)-w(x_{*})|\leq C(\|u\|_{L^{\infty}(B_{1}\cap\Omega)}+\|\overrightarrow{\mathbf{F}}\|_{L^{2q}(B_{1}\cap\Omega)}+\|f\|_{L^{q}(B_{1}\cap\Omega)}
)\hat{r}^{\frac{\alpha}{3}}\quad\text{in}~~\overline{B_{\hat{r}}(x_{*})}\cap B_{\frac{1}{4}}.
\end{equation}

\
\

{\bf Step 3:} Estimate $\|u-w\|_{L^{\infty}(B_{\frac{1}{4}-\delta}\cap\Omega)}$ where $\delta=\sqrt{\varepsilon_{1}+\varepsilon_{2}}$.

\
\

We consider $u-w$ in $B_{\frac{1}{4}-\delta}\cap\Omega$ where $\delta$ is sufficiently small satisfying $\delta=\sqrt{\varepsilon_{1}+\varepsilon_{2}}$ and we have
\begin{eqnarray*}
\left\{
\begin{array}{rcll}
-D_{j}(a_{ij}D_{i}(u-w))+b_{i}D_{i}(u-w)&=&f-\text{div}\overrightarrow{\mathbf{F}}+
D_{j}((a_{ij}-\delta_{ij})D_{i}w)-b_{i}D_{i}w\quad&\text{in}~~B_{\frac{1}{4}-\delta}\cap\Omega,\\
u-w&=&u-w\quad&\text{on}~~\partial B_{\frac{1}{4}-\delta}\cap\Omega, \\
u-w&=&g-w\quad&\text{on}~~B_{\frac{1}{4}-\delta}\cap\partial\Omega, \\
\end{array}
\right.
\end{eqnarray*}
Then by Theorem $\ref{lemma1}$, take $p=2q$, we get
\begin{eqnarray*}
\|u-w\|_{L^{\infty}(B_{\frac{1}{4}-\delta}\cap\Omega)}&\leq&\|u-w\|_{L^{\infty}(\partial B_{\frac{1}{4}-\delta}\cap\Omega)}
+\|g-w\|_{L^{\infty}(B_{\frac{1}{4}-\delta}\cap\partial\Omega)}\\
&&+C\{\|\overrightarrow{\mathbf{F}}\|_{L^{2q}(B_{\frac{1}{4}-\delta}\cap\Omega)}
+\|f\|_{L^{q}(B_{\frac{1}{4}-\delta}\cap\Omega)}\\
&&+\|(a_{ij}-\delta_{ij})D_{i}w\|_{L^{2q}(B_{\frac{1}{4}-\delta}\cap\Omega)}
+\|b_{i}D_{i}w\|_{L^{q}(B_{\frac{1}{4}-\delta}\cap\Omega)}\}.
\end{eqnarray*}

Next we estimate the right terms of above inequality.

To estimate of $\|u-w\|_{L^{\infty}(\partial B_{\frac{1}{4}-\delta}\cap\Omega)}$, we divide $\partial B_{\frac{1}{4}-\delta}\cap\Omega$ into three parts: $\partial B_{\frac{1}{4}-\delta}\cap\{x_{n}\geq\varepsilon+\delta^{\frac{1}{3}}\}$, $\partial B_{\frac{1}{4}-\delta}\cap\{\varepsilon\leq x_{n}\leq\varepsilon+\delta^{\frac{1}{3}}\}$, $\partial B_{\frac{1}{4}-\delta}\cap\{-\varepsilon\leq x_{n}\leq\varepsilon\}\cap\Omega$. For $x\in\partial B_{\frac{1}{4}-\delta}\cap\{x_{n}\geq\varepsilon\}$, there exists $x_{*}\in\partial B_{\frac{1}{4}}\cap\{x_{n}\geq\varepsilon\}$ such that $x\in\partial B_{\delta}(x_{*})$. Taking $\hat{r}=\delta=\sqrt{\varepsilon_{1}+\varepsilon_{2}}$ in $(\ref{wf})$ we can get for each $x\in\partial B_{\frac{1}{4}-\delta}\cap\{x_{n}\geq\varepsilon+\delta^{\frac{1}{3}}\}$,
\begin{eqnarray*}
|w(x)-w(x_{*})|\leq C(\|u\|_{L^{\infty}(B_{1}\cap\Omega)}+\|\overrightarrow{\mathbf{F}}\|_{L^{2q}(B_{1}\cap\Omega)}
+\|f\|_{L^{q}(B_{1}\cap\Omega)})\delta^{\frac{\alpha}{3}}.
\end{eqnarray*}
When $x\in\partial B_{\frac{1}{4}-\delta}\cap\{\varepsilon\leq x_{n}\leq\varepsilon+\delta^{\frac{1}{3}}\}$, then $x_{*}$ corresponding to $x$ belongs to $\partial B_{\frac{1}{4}}\cap\{\varepsilon\leq x_{n}\leq\varepsilon+\delta^{\frac{1}{3}}+\delta\}$, by $(\ref{wg})$ we have
$$|w(x)|\leq C^{*}(\varepsilon+\delta^{\frac{1}{3}}+\varepsilon)^{\alpha}+3\|g\|_{L^{\infty}(B_{1}\cap\partial\Omega)},
$$
$$|w(x_{*})|\leq C^{*}(\varepsilon+\delta^{\frac{1}{3}}+\delta+\varepsilon)^{\alpha}+3\|g\|_{L^{\infty}(B_{1}\cap\partial\Omega)}.
$$
It follows that
\begin{eqnarray*}
|w(x)-w(x_{*})|&\leq& 2^{\alpha+1}C^{*}(\varepsilon+\delta^{\frac{1}{3}})^{\alpha}+6\|g\|_{L^{\infty}(B_{1}\cap\partial\Omega)}\\
&\leq&2^{\alpha+1}C^{*}(\varepsilon^{\alpha}+\delta^{\frac{\alpha}{3}})+6\|g\|_{L^{\infty}(B_{1}\cap\partial\Omega)}.
\end{eqnarray*}
Then we obtain the estimation of $\|u-w\|_{L^{\infty}(\partial B_{\frac{1}{4}-\delta}\cap\{x_{n}\geq\varepsilon\})}$, i.e. by using $(\ref{ubr})$, we have for each $x\in\partial B_{\frac{1}{4}-\delta}\cap\{x_{n}\geq\varepsilon\}$,
\begin{eqnarray*}
& & |u(x)-w(x)|\\
 &\leq &  |u(x)-u(x_{*})|+|w(x)-w(x_{*})|\\
&\leq&C(\|u\|_{L^{\infty}(B_{1}\cap\Omega)}+\|\overrightarrow{\mathbf{F}}\|_{L^{2q}(B_{1}\cap\Omega)}+\|f\|_{L^{q}(B_{1}\cap\Omega)}
+\|g\|_{L^{\infty}(B_{1}\cap\partial\Omega)})(\delta^{\frac{\alpha}{3}}+\varepsilon^{\alpha})+6\|g\|_{L^{\infty}(B_{1}\cap\partial\Omega)}.
\end{eqnarray*}
When $x\in\partial B_{\frac{1}{4}-\delta}\cap\{-\varepsilon \leq x_{n}\leq\varepsilon\}\cap\Omega$, combining $(\ref{uep})$ with $(\ref{w1})$ we have
\begin{eqnarray*}& & |u(x)-w(x)| \\
 & \leq & C(\|u\|_{L^{\infty}(B_{1}\cap\Omega)}+\|\overrightarrow{\mathbf{F}}\|_{L^{2q}(B_{1}\cap\Omega)}+\|f\|_{L^{q}(B_{1}\cap\Omega)}
+\|g\|_{L^{\infty}(B_{1}\cap\partial\Omega)})(2\varepsilon)^{\alpha}+6\|g\|_{L^{\infty}(B_{1}\cap\partial\Omega)}.
\end{eqnarray*}
Putting the above estimates together, we have
\begin{eqnarray*}& &
 \|u-w\|_{L^{\infty}(\partial B_{\frac{1}{4}-\delta}\cap\Omega)} \\
 & \leq &
C(\|u\|_{L^{\infty}(B_{1}\cap\Omega)}+\|\overrightarrow{\mathbf{F}}\|_{L^{2q}(B_{1}\cap\Omega)}+\|f\|_{L^{q}(B_{1}\cap\Omega)}
+\|g\|_{L^{\infty}(B_{1}\cap\partial\Omega)})(\delta^{\frac{\alpha}{3}}+\varepsilon^{\alpha})+12\|g\|_{L^{\infty}(B_{1}\cap\partial\Omega)}.
\end{eqnarray*}

 To estimate of $\|g-w\|_{L^{\infty}(B_{\frac{1}{4}-\delta}\cap\partial\Omega)}$, duo to $B_{\frac{1}{4}-\delta}\cap\partial\Omega\subset \{-\varepsilon \leq x_{n}\leq\varepsilon\}$, it follows by  $(\ref{w1})$ that
 \begin{eqnarray*}
&& \|g-w\|_{L^{\infty}(B_{\frac{1}{4}-\delta}\cap\partial\Omega)} \\
&\leq &  C(\|u\|_{L^{\infty}(B_{1}\cap\Omega)}+\|\overrightarrow{\mathbf{F}}\|_{L^{2q}(B_{1}\cap\Omega)}+\|f\|_{L^{q}(B_{1}\cap\Omega)}
+\|g\|_{L^{\infty}(B_{1}\cap\partial\Omega)})(2\varepsilon)^{\alpha}+6\|g\|_{L^{\infty}(B_{1}\cap\partial\Omega)}.
\end{eqnarray*}

To estimate of $\|(a_{ij}-\delta_{ij})D_{i}w\|_{L^{2n}(B_{\frac{1}{4}-\delta}\cap\Omega)}$ and $\|b_{i}D_{i}w\|_{L^{n}(B_{\frac{1}{4}-\delta}\cap\Omega)}$, we notice that
$$\|(a_{ij}-\delta_{ij})D_{i}w\|_{L^{2q}(B_{\frac{1}{4}-\delta}\cap\Omega)}\leq  C \|a_{ij}-\delta_{ij}\|_{L^{\infty}(B_{1}\cap\Omega)}\|Dw\|_{L^{\infty}(B_{\frac{1}{4}-\delta}\cap\Omega)}
\leq C \varepsilon_{1}\|Dw\|_{L^{\infty}(B_{\frac{1}{4}-\delta}\cap\Omega)},
$$
$$\|b_{i}D_{i}w\|_{L^{q}(B_{\frac{1}{4}-\delta}\cap\Omega)}\leq \|b_{i}\|_{L^{q}(B_{1}\cap\Omega)}\|Dw\|_{L^{\infty}(B_{\frac{1}{4}-\delta}\cap\Omega)}\leq\varepsilon_{2}\|Dw\|_{L^{\infty}(B_{\frac{1}{4}-\delta}\cap\Omega)}.
$$
By the property of harmonic functions and $w$ we given, we have
$$\|Dw\|_{L^{\infty}(B_{\frac{1}{4}-\delta}\cap\Omega)}\leq \frac{C}{\delta}\|w\|_{L^{\infty}(B_{\frac{1}{4}}\cap\{x_{n}>-\varepsilon\})}
\leq \frac{C}{\delta}\|u\|_{L^{\infty}(B_{1}\cap\Omega)}.
$$
It follows that
\begin{eqnarray*} & & \|(a_{ij}-\delta_{ij})D_{i}w\|_{L^{2q}(B_{\frac{1}{4}-\delta}\cap\Omega)}+
\|b_{i}D_{i}w\|_{L^{q}(B_{\frac{1}{4}-\delta}\cap\Omega)}\\
&\leq& \frac{C(\varepsilon_{1}+\varepsilon_{2})}{\delta}\|u\|_{L^{\infty}(B_{1}\cap\Omega)}\\
&\leq&C\sqrt{\varepsilon_{1}+\varepsilon_{2}}\|u\|_{L^{\infty}(B_{1}\cap\Omega)}.\\
\end{eqnarray*}
Finally we can get the estimation of $\|u-w\|_{L^{\infty}(B_{\frac{1}{4}-\delta}\cap\Omega)}$, that is
\begin{eqnarray*}
\|u-w\|_{L^{\infty}(B_{\frac{1}{4}-\delta}\cap\Omega)}&\leq&C(\varepsilon^{\alpha}
+(\varepsilon_{1}+\varepsilon_{2})^{\frac{\alpha}{6}}+\sqrt{\varepsilon_{1}+\varepsilon_{2}})\|u\|_{L^{\infty}(B_{1}\cap\Omega)}\\
&&+C\|g\|_{L^{\infty}(B_{1}\cap\partial\Omega)}
+C(\|\overrightarrow{\mathbf{F}}\|_{L^{2q}(B_{1}\cap\Omega)}+\|f\|_{L^{q}(B_{1}\cap\Omega)}).
\end{eqnarray*}

\
\

{\bf Step 4:} Approximate $w$ by a harmonic function $h$.

\
\

The following proof is similar to \cite{17}. In the following we denote $\|u\|_{L^{\infty}(B_{1}\cap\Omega)}$ by $\mu$. Let $\Gamma$ be defined for $x\in\mathbb{R}^{n}\setminus\{0\}$ by
\begin{eqnarray*} \Gamma(x)=\Gamma(|x|)=
\begin{cases}
-\frac{1}{2\pi}\ln|x|, \quad &n=2,\\
\frac{1}{(n-2)\omega_{n}}|x|^{2-n},\quad &n\geq3,\\
\end{cases}
\end{eqnarray*}
where $\omega_{n}$ is the surface area of the unit sphere in $\mathbb{R}^{n}$. This function $\Gamma$ is usually called the fundamental solution of the Laplace operator. By a simple calculation, we have $\Delta\Gamma=0$ in $\mathbb{R}^{n}\setminus\{0\}$. Then for any $x=(x',x_{n})\in\overline{B_{\frac{1}{16}}}\cap\{x_{n}\geq-\varepsilon\}$, we take $(x',0)\in T_{\frac{1}{16}}$, we consider a function
$$l(s)=\frac{\Gamma\left(s-\left(x',-\frac{1}{16}-\varepsilon\right)\right)-\Gamma\left(\frac{1}{16}\right)}{\Gamma\left(\frac{3}{16}\right)
-\Gamma\left(\frac{1}{16}\right)}\mu.
$$
Clearly $l(s)$ is harmonic between $B_{\frac{1}{16}}(x',-\frac{1}{16}-\varepsilon)$ and $B_{\frac{3}{16}}(x',-\frac{1}{16}-\varepsilon)$ and
\begin{equation}\label{5}
\left\{
\begin{array}{rcll}
&l(s)&=0\qquad&\text{on}~~\partial B_{\frac{1}{16}}(x',-\frac{1}{16}-\varepsilon),\\
0<&l(s)&<\mu\qquad&\text{between}~B_{\frac{1}{16}}(x',-\frac{1}{16}-\varepsilon)~\text{and}~B_{\frac{3}{16}}(y',-\frac{1}{16}-\varepsilon), \\
&l(s)&=\mu\qquad&\text{on}~~\partial B_{\frac{3}{16}}(x',-\frac{1}{16}-\varepsilon).\\
\end{array}
\right.
\end{equation}
From (\ref{w}) and (\ref{5}) we get $w-l$ satisfies
\begin{eqnarray*}
\left\{
\begin{array}{rcll}
\Delta (w-l)&=&0\qquad&\text{in}~~B_{\frac{3}{16}}(x',-\frac{1}{16}-\varepsilon)\cap\{x_{n}>-\varepsilon\},\\
w-l&\leq&0\qquad&\text{on}~~\partial B_{\frac{3}{16}}(x',-\frac{1}{16}-\varepsilon)\cap\{x_{n}>-\varepsilon\}, \\
w-l&\leq&0\qquad&\text{on}~~B_{\frac{3}{16}}(x',-\frac{1}{16}-\varepsilon)\cap\{x_{n}=-\varepsilon\}.\\
\end{array}
\right.
\end{eqnarray*}
Applying the maximum principle, it yields that
$$w\leq l \quad \text{in}~~B_{\frac{3}{16}}(x',-\frac{1}{16}-\varepsilon)\cap\{x_{n}\geq-\varepsilon\}.
$$
Similarly, repeating the above process for $w+l$, it's easy to get that $$w\geq-l\quad \text{in}~B_{\frac{3}{16}}(x',-\frac{1}{16}-\varepsilon)\cap\{x_{n}\geq-\varepsilon\}.
$$
Consequently we have
$$|w|\leq l \quad \text{in}~~B_{\frac{3}{16}}(x',-\frac{1}{16}-\varepsilon)\cap\{x_{n}\geq-\varepsilon\}.
$$
Furthermore for arbitrary $x_{1}\in\partial B_{\frac{1}{16}}(x',-\frac{1}{16}-\varepsilon)$, in the radial direction, we have
$$\frac{l(s)-l(x_{1})}{|s-x_{1}|}\leq C_{0}\mu \quad \text{between}~B_{\frac{1}{16}}(x',-\frac{1}{16}-\varepsilon)~\text{and}~B_{\frac{3}{16}}(x',-\frac{1}{16}-\varepsilon).
$$
In particular for $x_{1}=(x',-\varepsilon)$, $s=x$, it follows that
$$l(x)\leq C_{0}\mu(x_{n}+\varepsilon).
$$
Since $x$ can be chosen in $\overline{B_{\frac{1}{16}}}\cap\{x_{n}\geq-\varepsilon\}$ arbitrarily, then
\begin{equation}\label{6}
|w(x)|\leq l(x)\leq C_{0}\mu(x_{n}+\varepsilon) \quad \text{in}~\overline{B_{\frac{1}{16}}}\cap\{x_{n}\geq-\varepsilon\}.
\end{equation}
Then (\ref{w}) and (\ref{6}) imply that $v$ satisfies the following conditions,
\begin{eqnarray*}
\left\{
\begin{array}{rcll}
\Delta w&=&0\qquad&\text{in}~~B_{\frac{1}{16}}\cap\{x_{n}>-\varepsilon\},\\
|w(x)|&\leq& C_{0}\mu(x_{n}+\varepsilon)\qquad&\text{in}~~ \overline{B_{\frac{1}{16}}}\cap\{x_{n}\geq-\varepsilon\}, \\
w&=&0\qquad&\text{on}~~B_{\frac{1}{16}}\cap\{x_{n}=-\varepsilon\}.\\
\end{array}
\right.
\end{eqnarray*}
Now it's time to find the harmonic function. We take $h$ be a harmonic function defined in $B_{\frac{1}{16}}$ which is odd with respect to $x_{n}$ and satisfies the following conditions,
\begin{eqnarray*}
\left\{
\begin{array}{rcll}
\Delta h&=&0\qquad&\text{in}~~B_{\frac{1}{16}}^{+},\\
h&=&0\qquad&\text{on}~~T_{\frac{1}{16}}, \\
h&=&w\qquad&\text{on}~~\partial B_{\frac{1}{16}}^{+}\cap\{x\in\mathbb{R}^{n}: x_{n}\geq\varepsilon\},\\
h&=&2C_{0}\mu\varepsilon\qquad&\text{on}~~\partial B_{\frac{1}{16}}^{+}\cap\{x\in\mathbb{R}^{n}: 0<x_{n}<\varepsilon\}.\\
\end{array}
\right.
\end{eqnarray*}
Applying the maximum principle to $h$, we get
\begin{equation}\label{7}|h(x)|\leq C_{0}\mu(x_{n}+\varepsilon)+2C_{0}\mu\varepsilon \quad \text{in}~~~\overline{B_{\frac{1}{16}}^{+}},
\end{equation}
$$\|h\|_{L^{\infty}(B_{\frac{1}{16}}^{+})}\leq(1+2C_{0}\varepsilon)\mu.$$
Next we consider $w-h$ in $B_{\frac{1}{16}}^{+}\cap\Omega$ to obtain
\begin{eqnarray*}
\left\{
\begin{array}{rcll}
&\Delta (w-h)&=0\qquad&\text{in}~~B_{\frac{1}{16}}^{+}\cap\Omega,\\
&w-h&=0\qquad&\text{on}~~\partial B_{\frac{1}{16}}^{+}\cap\{x_{n}\geq\varepsilon\},\\
-4C_{0}\mu\varepsilon\leq &w-h&\leq 0\qquad&\text{on}~~\partial B_{\frac{1}{16}}^{+}\cap\{0<x_{n}<\varepsilon\},\\
-6C_{0}\mu\varepsilon\leq&w-h&\leq6C_{0}\mu\varepsilon\qquad&\text{on}~~B_{\frac{1}{16}}^{+}\cap\partial\Omega,\\
-C_{0}\mu\varepsilon\leq&w-h&\leq C_{0}\mu\varepsilon\qquad&\text{on}~~T_{\frac{1}{16}}\cap\Omega. \\
\end{array}
\right.
\end{eqnarray*}
Using the maximum principle again we obtain
$$|w-h|\leq6C_{0}\mu\varepsilon\quad\text{in}~~B_{\frac{1}{16}}^{+}\cap\Omega.$$
Since $h$ is odd with respect to $x_{n}$ and $B_{\frac{1}{16}}\cap\Omega\subset B_{\frac{1}{16}}\cap\{x\in\mathbb{R}^{n}: x_{n}>-\varepsilon\}$ for some $0<\varepsilon<\frac{1}{16}$, it's easy to get $|h|\leq4C_{0}\mu\varepsilon$ in $B_{\frac{1}{16}}^{-}\cap\Omega.$
Combining with (\ref{6}) we get
$$|w-h|\leq5C_{0}\mu\varepsilon\quad \text{in}~~B_{\frac{1}{16}}^{-}\cap\Omega.$$
From above two inequalities we get
\begin{equation}\label{10}\left\|w-h\right\|_{L^{\infty}(B_{\frac{1}{16}}\cap\Omega)}\leq6C_{0}\mu\varepsilon.
\end{equation}

Then from the estimation of $\|u-w\|_{L^{\infty}(B_{\frac{1}{4}-\delta}\cap\Omega)}$ and (\ref{10}), we can get the following desired result by the triangle inequality,
\begin{eqnarray*}
\left\|u-h\right\|_{L^{\infty}(B_{\frac{1}{16}}\cap\Omega)}&\leq&
\|u-w\|_{L^{\infty}(B_{\frac{1}{4}-\delta}\cap\Omega)}+\|w-h\|_{L^{\infty}(B_{\frac{1}{16}}\cap\Omega)}\\
&\leq&C(\varepsilon^{\alpha}
+(\varepsilon_{1}+\varepsilon_{2})^{\frac{\alpha}{6}}+\sqrt{\varepsilon_{1}+\varepsilon_{2}})\|u\|_{L^{\infty}(B_{1}\cap\Omega)}\\
&&+C\|g\|_{L^{\infty}(B_{1}\cap\partial\Omega)}
+C(\|\overrightarrow{\mathbf{F}}\|_{L^{2q}(B_{1}\cap\Omega)}+\|f\|_{L^{q}(B_{1}\cap\Omega)})\\
&&+6C_{0}\varepsilon\|u\|_{L^{\infty}(B_{1}\cap\Omega)}\\
&\leq& C\|g\|_{L^{\infty}(B_{1}\cap\partial\Omega)}+C(\|\overrightarrow{\mathbf{F}}\|_{L^{2q}(B_{1}\cap\Omega)}+\|f\|_{L^{q}(B_{1}\cap\Omega)})\\
&&+C(\varepsilon^{\alpha}+(\varepsilon_{1}+\varepsilon_{2})^{\frac{\alpha}{6}}+\sqrt{\varepsilon_{1}+\varepsilon_{2}})\|u\|_{L^{\infty}(B_{1}\cap\Omega)}.
\end{eqnarray*}
\end{proof}
\begin{rem}\label{rem2.3}
In fact, $x_{n}$ can be regarded as $x\cdot\vec{e}_{n}$. Therefore in Lemma \ref{fs}, $x_{n}$ can be replaced by $x\cdot\vec{n}$ for arbitrary unit vector $\vec{n}$, i.e Lemma \ref{fs} also holds when $\Omega$ satisfying $$ B_{1}\cap\{x\in\mathbb{R}^{n}: x\cdot\vec{n}>\varepsilon\}\subset B_{1}\cap\Omega\subset B_{1}\cap\{x\in\mathbb{R}^{n}: x\cdot\vec{n}>-\varepsilon\}$$
 for some $0<\varepsilon<\frac{1}{16}$.
\end{rem}
\begin{lm}[Key lemma]\label{kl}
Let $\Omega$, $a_{ij}$ and $b_{i}$ satisfy the assumptions in Lemma $\ref{fs}$, then there exists $\alpha>0$, $1>\lambda>0$ and universal constants $\widetilde{C}, C_{1}, C_{2}, C_{3}>0$ such that for any functions $\overrightarrow{\mathbf{F}}(x,u)\in L^{\infty}(B_{1}\cap\Omega)$, $f\in L^{q}(B_{1}\cap\Omega)$, $g\in L^{\infty}(B_{1}\cap\partial\Omega)$, if $u$ is the solution of
\begin{eqnarray*}
\left\{
\begin{array}{rcll}
-D_{j}(a_{ij}D_{i}u)+b_{i}D_{i}u&=&f-\text{div}\overrightarrow{\mathbf{F}}(x,u)\qquad&\text{in}~~B_{1}\cap\Omega,\\
u&=&g\qquad&\text{on}~~B_{1}\cap\partial\Omega, \\
\end{array}
\right.
\end{eqnarray*}
and $v$ is a Lipschitz solution of
$$
\Delta v=\text{div}\overrightarrow{\mathbf{F}}(x,g(0))\quad \text{in}~~B_{1}
$$
with Lipschitz constant $T$, then there exists a constant K such that
\begin{eqnarray*}
\|u-v-Kx_{n}\|_{L^{\infty}(B_{\lambda}\cap\Omega)}&\leq& C_{1}\|g-v\|_{L^{\infty}(B_{1}\cap\partial\Omega)}
+C_{2}(\lambda^{2}+\varepsilon^{\alpha}+(\varepsilon_{1}+
\varepsilon_{2})^{\frac{\alpha}{6}}+\sqrt{\varepsilon_{1}+\varepsilon_{2}})\|u-v\|_{L^{\infty}(B_{1}\cap\Omega)}\\
&&+C_{3}\left(\|\overrightarrow{\mathbf{F}}(x,u)-\overrightarrow{\mathbf{F}}(x,g(0))\|_{L^{\infty}(B_{1}\cap\Omega)}
+\|f\|_{L^{q}(B_{1}\cap\Omega)}+T(\varepsilon_{1}+\varepsilon_{2})\right),
\end{eqnarray*}
and
$$0<|K|\leq\widetilde{C}\|u-v\|_{L^{\infty}(B_{1}\cap\Omega)}.
$$
\end{lm}
\begin{proof}
By the definition of $u$ and $v$ we get
$$
-D_{j}(a_{ij}D_{i}(u-v))+b_{i}D_{i}(u-v)=
f-\text{div}(\overrightarrow{\mathbf{F}}(x,u)-\overrightarrow{\mathbf{F}}(x,g(0))+
D_{j}((a_{ij}-\delta_{ij})D_{i}v)-b_{i}D_{i}v$$ in $ B_{1}\cap\Omega$, and $$ u-v = g-v$$ on $B_{1}\cap\partial\Omega$.
Then by Lemma $\ref{fs}$, for some $\alpha>0$, there exists a universal constant $C_{0}$ and a harmonic function $h$ defined in $B_{\frac{1}{16}}$ which is odd with respect to $x_{n}$ satisfying
$$\|h\|_{L^{\infty}(B_{\frac{1}{16}})}\leq(1+2C_{0}\varepsilon)\|u-v\|_{L^{\infty}(B_{1}\cap\Omega)}$$
such that
\begin{equation}\label{11}
\begin{aligned}
\|u-v-h\|_{L^{\infty}(B_{\frac{1}{16}}\cap\Omega)}&\leq C\|g-v\|_{L^{\infty}(B_{1}\cap\partial\Omega)}
+C(\|\overrightarrow{\mathbf{F}}(x,u)-\overrightarrow{\mathbf{F}}(x,g(0))\|_{L^{2q}(B_{1}\cap\Omega)}+\|f\|_{L^{q}(B_{1}\cap\Omega)})\\
&\quad+C\left(\|(a_{ij}-\delta_{ij})D_{i}v\|_{L^{2q}(B_{1}\cap\Omega)}
+\|b_{i}D_{i}v\|_{L^{q}(B_{1}\cap\Omega)}\right)\\
&\quad+C(\varepsilon^{\alpha}+(\varepsilon_{1}+\varepsilon_{2})^{\frac{\alpha}{6}}+\sqrt{\varepsilon_{1}+\varepsilon_{2}})\|u-v\|_{L^{\infty}(B_{1}\cap\Omega)}\\
&\leq C\|g-v\|_{L^{\infty}(B_{1}\cap\partial\Omega)}
+C(\|\overrightarrow{\mathbf{F}}(x,u)-\overrightarrow{\mathbf{F}}(x,g(0))\|_{L^{\infty}(B_{1}\cap\Omega)}+\|f\|_{L^{q}(B_{1}\cap\Omega)})\\
&\quad+CT\left(\|a_{ij}-\delta_{ij}\|_{L^{\infty}(B_{1}\cap\Omega)}
+\|b_{i}\|_{L^{q}(B_{1}\cap\Omega)}\right)\\
&\quad+C(\varepsilon^{\alpha}+(\varepsilon_{1}+\varepsilon_{2})^{\frac{\alpha}{6}}+\sqrt{\varepsilon_{1}+\varepsilon_{2}})\|u-v\|_{L^{\infty}(B_{1}\cap\Omega)}.\\
\end{aligned}
\end{equation}

Take $L$ be the first order Taylor polynomial of $h$ at 0, i.e. $L(x)=Dh(0)\cdot x+h(0).$ Then there exists $\xi\in B_{\frac{1}{32}}$ such that for $|x|\leq\frac{1}{32},$
\begin{equation}\label{12}
|h(x)-L(x)|\leq\frac{1}{2}|D^{2}h(\xi)||x|^{2}.
\end{equation}
Since $h=0$ on $B_{\frac{1}{16}}\cap\{x_{n}=0\},$ then $L(x)=Kx_{n}$, where $|K|=|Dh(0)|$.
Note that $h$ is a harmonic function which is odd with respect to $x_{n}$ in $B_{\frac{1}{16}}$, according to the property of harmonic function, when $|x|\leq\frac{1}{32},$
$$|D^{2}h(x)|+|Dh(x)|\leq A\|h\|_{L^{\infty}(B_{\frac{1}{16}})}\leq A(1+2C_{0}\varepsilon)\|u-v\|_{L^{\infty}(B_{1}\cap\Omega)},
$$
where $A$ is a constant depending only on $n$. It follows that
$$|D^{2}h(\xi)|+|K|\leq A(1+2C_{0}\varepsilon)\|u-v\|_{L^{\infty}(B_{1}\cap\Omega)}.
$$
Finally, combining (\ref{11}) with (\ref{12}), if we take $0<\lambda<\frac{1}{32},$ then we have
\begin{eqnarray*}
\|u-v-Kx_{n}\|_{L^{\infty}(B_{\lambda}\cap\Omega)}&\leq&\|u-v-h\|_{L^{\infty}(B_{\lambda}\cap\Omega)}+\|h-L\|_{L^{\infty}(B_{\lambda}\cap\Omega)}\\
&\leq& C\|g-v\|_{L^{\infty}(B_{1}\cap\partial\Omega)}
+C(\|\overrightarrow{\mathbf{F}}(x,u)-\overrightarrow{\mathbf{F}}(x,g(0))\|_{L^{\infty}(B_{1}\cap\Omega)}+\|f\|_{L^{q}(B_{1}\cap\Omega)})\\
&&+CT\left(\|a_{ij}-\delta_{ij}\|_{L^{\infty}(B_{1}\cap\Omega)}
+\|b_{i}\|_{L^{q}(B_{1}\cap\Omega)}\right)\\
&&+C(\varepsilon^{\alpha}+(\varepsilon_{1}+\varepsilon_{2})^{\frac{\alpha}{6}}+\sqrt{\varepsilon_{1}+\varepsilon_{2}})\|u-v\|_{L^{\infty}(B_{1}\cap\Omega)}\\
&&+\frac{1}{2}\lambda^{2}A(1+2C_{0}\varepsilon)\|u-v\|_{L^{\infty}(B_{1}\cap\Omega)}\\
&\leq& C_{1}\|g-v\|_{L^{\infty}(B_{1}\cap\partial\Omega)}+C_{2}(\lambda^{2}+\varepsilon^{\alpha}+(\varepsilon_{1}+
\varepsilon_{2})^{\frac{\alpha}{6}}+\sqrt{\varepsilon_{1}+\varepsilon_{2}})\|u-v\|_{L^{\infty}(B_{1}\cap\Omega)}\\
&&+C_{3}\left(\|\overrightarrow{\mathbf{F}}(x,u)-\overrightarrow{\mathbf{F}}(x,g(0))\|_{L^{\infty}(B_{1}\cap\Omega)}
+\|f\|_{L^{q}(B_{1}\cap\Omega)}+T(\varepsilon_{1}+\varepsilon_{2})\right)
\end{eqnarray*}
\end{proof}
\begin{rem}\label{rem3.1}
Let $\Omega$, $a_{ij}$ and $b_{i}$ satisfy the assumptions in  Lemma $\ref{fs}$, then there exists $\alpha>0$, $0<\lambda<1$ and universal constants $\widetilde{C}, C_{1}, C_{2}, C_{3}>0$ such that for any functions $\overrightarrow{\mathbf{F}}(x)\in L^{\infty}(B_{1}\cap\Omega)$, $f\in L^{q}(B_{1}\cap\Omega)$, $g\in L^{\infty}(B_{1}\cap\partial\Omega)$, if $u$ is the solution of
\begin{eqnarray*}
\left\{
\begin{array}{rcll}
-D_{j}(a_{ij}D_{i}u)+b_{i}D_{i}u&=&f-\text{div}\overrightarrow{\mathbf{F}}(x)\qquad&\text{in}~~B_{1}\cap\Omega,\\
u&=&g\qquad&\text{on}~~B_{1}\cap\partial\Omega, \\
\end{array}
\right.
\end{eqnarray*}
then there exists a constant K such that
\begin{eqnarray*}
\|u-Kx_{n}\|_{L^{\infty}(B_{\lambda}\cap\Omega)}&\leq& C_{1}\|g\|_{L^{\infty}(B_{1}\cap\partial\Omega)}
+C_{2}(\lambda^{2}+\varepsilon^{\alpha}+(\varepsilon_{1}+
\varepsilon_{2})^{\frac{\alpha}{6}}+\sqrt{\varepsilon_{1}+\varepsilon_{2}})\|u\|_{L^{\infty}(B_{1}\cap\Omega)}\\
&&+C_{3}(\|\overrightarrow{\mathbf{F}}\|_{L^{\infty}(B_{1}\cap\Omega)}+\|f\|_{L^{q}(B_{1}\cap\Omega)}),
\end{eqnarray*}
and
$$0<|K|\leq\widetilde{C}\|u\|_{L^{\infty}(B_{1}\cap\Omega)}.
$$
\end{rem}
\begin{rem}\label{rem2.7}
Similar to Remark $\ref{rem2.3}$, $x_{n}$ can also be regarded as $x\cdot\vec{e}_{n}$ and can be substituted by $x\cdot\vec{n}$ for arbitrary unit vector $\vec{n}$ in Lemma $\ref{kl}$ and Remark $\ref{rem3.1}$.
\end{rem}

\section{Boundary Lipschitz regularity under $C^{1,\text{Dini}}$ condition}
In this section we will prove Theorem \ref{mr1}. We divide this proof into four steps, which are similar to Section 3 in \cite{17}. Based on the key lemma in Section 2, we begin to iterate and approximate $u$ by a Lipschitz function $v$ and linear functions in different scales. Finally we will prove that the sum of errors from different scales is convergent. This step can reflect how the Dini conditions are applied.

Before the proof, we first simplify the problem. In fact, we can assume $x_{0}=0$ is a boundary point and we only need to prove the boundary Lipschitz regularity at 0. For convenience, we can choose an appropriate coordinate system such that $\vec{n}$ in Definition $\ref{boundarydef}$ is along the positive $x_{n}$-axis in the proof. So by definition, if $\partial\Omega$ is $C^{1,\text{Dini}}$ at $0,$ then for any $0<r\leq r_{0},$ $B_{r}\cap\partial\Omega\subset B_{r}\cap\left\{|x_{n}|\leq r\omega(r)\right\}$. We denote $v_{0}$ by $v$ and assume that
$$u(0)=g(0)=0,\quad v(0)=0, \quad r_{0}=1,
$$
$$\int_{0}^{1} \frac{\omega(r)}{r} d r \leq 1, \quad \int_{0}^{1} \frac{\sigma(r)}{r} d r \leq 1, \quad \int_{0}^{1} \frac{\varphi(r)}{r} d r \leq 1, \quad \int_{0}^{1} \frac{\omega_{1}(r)}{r} d r \leq 1,
$$
$$\omega_{2}(r)=\varepsilon_{2}r^{1-\frac{n}{q}},\quad \|b_{i}\|_{L^{q}(B_{1}\cap\Omega)}\leq\varepsilon_{2}=\omega_{2}(1),
$$
\begin{equation}\label{omega}
\omega(1)\leq\lambda^{\frac{2}{\alpha}}, \quad \max\{(\omega_{1}(1)+\omega_{2}(1))^{\frac{\alpha}{6}},(\omega_{1}(1)+\omega_{2}(1))^{\frac{1}{2}}\}\leq \lambda^{2}
\end{equation}
where $\varepsilon_{2}<1-\frac{n}{q}$ is small enough in Lemma $\ref{fs}$, $\alpha$ is determined in Lemma $\ref{kl}$ and $\lambda$ is small enough and satisfies
\begin{equation}\label{lambda}
0<\lambda<\frac{1}{32},\quad 4C_{2}\lambda<\frac{1}{4},
\end{equation}
$C_{2}$ is the constant in Lemma $\ref{kl}$ and Remark $\ref{rem3.1}$.

Besides, we can also assume $\vec{a}=0$ in definition \ref{bvdini}, if not, we can consider function $u^{*}:=u-\vec{a}\cdot x$, $\vec{a}=(a_{1},a_{2},\cdots,a_{n})$. We set $g^{*}=g-\vec{a}\cdot x$, then $u^{*}$ satisfies
\begin{eqnarray*}
\left\{
\begin{array}{rcll}
-D_{j}(a_{ij}D_{i}u^{*})+b_{i}D_{i}u^{*}&=&-\text {div}\overrightarrow{\mathbf{F}}(x,u^{*}+\vec{a}\cdot x)+D_{j}((a_{ij}-\delta_{ij})a_{i})-b_{i}a_{i}\qquad&\text{in}~~\Omega,\\
u^{*}&=&g^{*}\qquad&\text{on}~~\partial\Omega. \\
\end{array}
\right.
\end{eqnarray*}
We can also apply Lemmas in Section 2 to $u^{*}$ and repeat the four steps below to get the Lipschitz regularity of $u^{*}$ at 0.

Based on Lemma $\ref{kl}$ and Remark $\ref{rem3.1}$, the following lemma is an iteration result.
\begin{lm}\label{bmkm}
There exist sequences $\{N_{i}\}_{i=0}^{\infty}$ and nonnegative sequences $\{M_{i}\}_{i=0}^{\infty}$, $\{\xi_{i}\}_{i=0}^{\infty}$, $\{\eta_{i}\}_{i=0}^{\infty}$ with $N_{0}=0$, and for $i=0,1,2,\ldots,$
$$M_{i}=\|u-v-N_{i}x_{n}\|_{L^{\infty}(B_{\lambda^{i}}\cap\Omega)},
$$
$$\xi_{i}=C_{2}(\lambda^{2}+\omega(\lambda^{i})^{\alpha}+(\omega_{1}(\lambda^{i})+
\omega_{2}(\lambda^{i}))^{\frac{\alpha}{6}}+\sqrt{\omega_{1}(\lambda^{i})+
\omega_{2}(\lambda^{i})}),
$$
\begin{eqnarray*}
\eta_{i}=C_{1}\|g-v-N_{i}x_{n}\|_{L^{\infty}(B_{\lambda^{i}}\cap\partial\Omega)}+C_{3}\lambda^{i}\left(\|\overrightarrow{\mathbf{F}}(x,u)-\overrightarrow{\mathbf{F}}(x,0)\|_{L^{\infty}(\Omega_{\lambda^{i}})}
+(T+|N_{i}|)(\omega_{1}(\lambda^{i})+\omega_{2}(\lambda^{i}))\right),
\end{eqnarray*}
\begin{eqnarray*}\label{n}
|N_{i+1}-N_{i}|\leq\frac{\widetilde{C}}{\lambda^{i}}\|u-v-N_{i}x_{n}\|_{L^{\infty}(\Omega_{\lambda^{i}})},
\end{eqnarray*}
such that
\begin{equation}\label{induction}
M_{i+1}\leq \xi_{i}M_{i}+\eta_{i}.
\end{equation}
\end{lm}\begin{proof}
We prove this lemma inductively by using Remark $\ref{rem3.1}$ repeatedly.

When $i=0$, $N_{0}=0$. 
By Definition $\ref{boundarydef}$ and Remark $\ref{rem1.5}$, we have $B_{1}\cap\partial\Omega\subset\left\{|x_{n}|\leq \omega(1)\right\}$ and for any $x_{0}\in\partial\Omega$, $$\lim\limits_{r\rightarrow0}\frac{|B_{r}(x_{0})\cap\Omega^{c}|}{|B_{r}(x_{0})|}=\frac{1}{2}>0.
$$
Therefore by lemma $\ref{kl}$, combining with the conditions and assumptions on $a_{ij}$ and $b_{i}$, there exists $|N_{1}|\leq\widetilde{C}\|u-v\|_{L^{\infty}(B_{1}\cap\Omega)}$ such that
\begin{eqnarray*}
M_{1}&=&\|u-v-N_{1}x_{n}\|_{L^{\infty}(B_{\lambda}\cap\Omega)}\\&\leq&
C_{1}\|g-v\|_{L^{\infty}(B_{1}\cap\partial\Omega)}\\
&&+C_{2}(\lambda^{2}+\omega(1)^{\alpha}+(\omega_{1}(1)+\omega_{2}(1))^{\frac{\alpha}{6}}+\sqrt{\omega_{1}(1)+\omega_{2}(1)})\|u-v\|_{L^{\infty}(B_{1}\cap\Omega)}\\
&&+C_{3}\left(\|\overrightarrow{\mathbf{F}}(x,u)-\overrightarrow{\mathbf{F}}(x,0)\|_{L^{\infty}(B_{1}\cap\Omega)}
+(T+|N_{0}|)(\omega_{1}(1)+\omega_{2}(1))\right)\\
&=&\xi_{0}M_{0}+\eta_{0},
\end{eqnarray*}
and
$$|N_{1}-N_{0}|\leq\widetilde{C}\|u-v\|_{L^{\infty}(B_{1}\cap\Omega)}.$$

Next we assume that the conclusion is true for $i-1$, that is $M_{i}\leq\xi_{i-1}M_{i-1}+\eta_{i-1}$. We set $\hat{u}=u-v-N_{i}x_{n}$ and consider the equation
\begin{eqnarray*}
\left\{
\begin{array}{rcll}
-D_{j}(a_{ij}D_{i}\hat{u})+b_{i}D_{i}\hat{u}&=&h(x)\qquad&\text{in}~~B_{\lambda^{i}}\cap\Omega,\\
\hat{u}&=&g-v-N_{i}x_{n}\qquad&\text{on}~~B_{\lambda^{i}}\cap\partial\Omega, \\
\end{array}
\right.
\end{eqnarray*}
where $h(x)=-\text{div}(\overrightarrow{\mathbf{F}}(x,u)-\overrightarrow{\mathbf{F}}(x,0))
+D_{j}((a_{ij}-\delta_{ij})D_{i}(v+N_{i}x_{n}))
-b_{i}D_{i}(v+N_{i}x_{n})$.

For $z=(z_{1}, z_{2}, \cdots,z_{n})\in\mathbb{R}^{n}$ we set
$$\widetilde{u}(z)=\frac{u(\lambda^{i}z)}{\lambda^{i}},
\quad\widetilde{v}(z)=\frac{v(\lambda^{i}z)+N_{i}\lambda^{i}z_{n}}{\lambda^{i}},
\quad\widetilde{g}(z)=\frac{g(\lambda^{i} z)}{\lambda^{i}},
$$
$$\overrightarrow{\widetilde{\mathbf{F}}}(z)=\overrightarrow{\mathbf{F}}(\lambda^{i} z,u(\lambda^{i} z))-\overrightarrow{\mathbf{F}}(\lambda^{i} z,0),
$$
$$\widetilde{a_{ij}}(z)=a_{ij}(\lambda^{i} z),\quad \widetilde{b_{i}}(z)=\lambda^{i} b_{i}(\lambda^{i} z).
$$
Then $\widetilde{u}(z)-\widetilde{v}(z)$ is a solution of
\begin{eqnarray*}
\left\{
\begin{array}{rcll}
-D_{j}(\widetilde{a_{ij}}D_{i}(\widetilde{u}-\widetilde{v}))+\widetilde{b_{i}}D_{i}(\widetilde{u}-\widetilde{v})&=&
-\text{div}\overrightarrow{\widetilde{\mathbf{F}}}
+D_{j}((\widetilde{a_{ij}}-\delta_{ij})D_{i}\widetilde{v})
-\widetilde{b_{i}}D_{i}\widetilde{v}\qquad&\text{in}~~B_{1}\cap\widetilde{\Omega},\\
\widetilde{u}-\widetilde{v}&=&\widetilde{g}-\widetilde{v}\qquad&\text{on}~~B_{1}\cap\partial\widetilde{\Omega}, \\
\end{array}
\right.
\end{eqnarray*}
where $\widetilde{\Omega}=\{z:\lambda^{i}z\in\Omega\}.$ Therefore $B_{1}\cap\partial\widetilde{\Omega}\subset B_{1}\cap\left\{|z_{n}|\leq \omega(\lambda^{i})\right\}$. Combining with the conditions and assumptions on $a_{ij}$ and $b_{i}$, we also have
$$\|\widetilde{a_{ij}}-\delta_{ij}\|_{L^{\infty}(B_{1}\cap\widetilde{\Omega})}=
\|{a_{ij}}-\delta_{ij}\|_{L^{\infty}(B_{\lambda^{i}}\cap\Omega)}\leq\omega_{1}(\lambda^{i}),
$$
$$\|\widetilde{b_{i}}\|_{L^{q}(B_{1}\cap\widetilde{\Omega})}=(\lambda^{i})^{1-\frac{n}{q}}\|b_{i}\|_{L^{q}(B_{\lambda^{i}}\cap\Omega)}\leq\omega_{2}(\lambda^{i}).
$$
Then by Remark $\ref{rem3.1}$, there exists a constant $K$ such that
\begin{eqnarray*}
\|\widetilde{u}-\widetilde{v}-Kz_{n}\|_{L^{\infty}(B_{\lambda}\cap\Omega)}
&\leq& C_{1}\|\widetilde{g}-\widetilde{v}\|_{L^{\infty}(B_{1}\cap\partial\widetilde{\Omega})}\\
&&+C_{2}(\lambda^{2}+\omega(\lambda^{i})^{\alpha}+(\omega_{1}(\lambda^{i})+
\omega_{2}(\lambda^{i}))^{\frac{\alpha}{6}}+\sqrt{\omega_{1}(\lambda^{i})+
\omega_{2}(\lambda^{i})})\|\widetilde{u}-\widetilde{v}\|_{L^{\infty}(B_{1}\cap\widetilde{\Omega})}\\
&&+C_{3}\left(\|\overrightarrow{\widetilde{\mathbf{F}}}\|_{L^{\infty}(B_{1}\cap\widetilde{\Omega})}+
\|(\widetilde{a_{ij}}-\delta_{ij})D_{i}\widetilde{v}\|_{L^{\infty}(B_{1}\cap\widetilde{\Omega})}
+\|\widetilde{b_{i}}D_{i}\widetilde{v}\|_{L^{q}(B_{1}\cap\widetilde{\Omega})}\right),
\end{eqnarray*}
where $|K|\leq \widetilde{C}\|\widetilde{u}-\widetilde{v}\|_{L^{\infty}(B_{1}\cap\widetilde{\Omega})}=\frac{\widetilde{C}}{\lambda^{i}}\|u-v-N_{i}x_{n}\|_{L^{\infty}(\Omega_{\lambda^{i}})}.$
Let $N_{i+1}=N_{i}+K,$ scaling back, then we get
\begin{eqnarray*}
M_{i+1}&=&\|u-v-N_{i+1}x_{n}\|_{L^{\infty}(\Omega_{\lambda^{i+1}})}\\
&\leq&C_{1}\|g-v-N_{i}x_{n}\|_{L^{\infty}(B_{\lambda^{i}}\cap\partial\Omega)}\\
&&+C_{2}(\lambda^{2}+\omega(\lambda^{i})^{\alpha}+(\omega_{1}(\lambda^{i})+
\omega_{2}(\lambda^{i}))^{\frac{\alpha}{6}}+\sqrt{\omega_{1}(\lambda^{i})+
\omega_{2}(\lambda^{i})})\|u-v-N_{i}x_{n}\|_{L^{\infty}(B_{\lambda^{i}}\cap\Omega)}\\
&&+C_{3}\lambda^{i}\left(\|\overrightarrow{\mathbf{F}}(x,u)-\overrightarrow{\mathbf{F}}(x,0)\|_{L^{\infty}(\Omega_{\lambda^{i}})}
+(T+|N_{i}|)(\omega_{1}(\lambda^{i})+\omega_{2}(\lambda^{i}))\right)\\
&=&\xi_{i}M_{i}+\eta_{i},
\end{eqnarray*}
and
$$|N_{i+1}-N_{i}|=|K|\leq\frac{\widetilde{C}}{\lambda^{i}}\|u-v-N_{i}x_{n}\|_{L^{\infty}(\Omega_{\lambda^{i}})}.$$
This completes the proof of Lemma $\ref{bmkm}$.
\end{proof}

The following three lemmas are similar to \cite{17}.
\begin{lm}\label{con}
$\sum\limits_{i=0}^{\infty} \frac{M_{i}}{\lambda^{i}}<\infty$ and $\lim \limits_{i \rightarrow \infty} N_{i}$ exists. We set
$$
\lim _{i \rightarrow \infty} N_{i}=\tau.
$$
\end{lm}
\begin{proof}
Since $T$ is the Lipschitz constant respect to $v$ and $v(0)=0$, then
$$\|v\|_{L^{\infty}(B_{\lambda^{i}}\cap\partial\Omega)}=\|v-v(0)\|_{L^{\infty}(B_{\lambda^{i}}\cap\partial\Omega)}\leq T\lambda^{i}.
$$
For $k\geq0,$ we suppose $P_{k}=\sum\limits_{i=0}^{k}\frac{M_{i}}{\lambda^{i}}.$ By Lemma $\ref{bmkm}$, noting that $N_{0}=0$, $M_{0}=\|u-v\|_{L^{\infty}(\Omega_{1})}$, then for any $k\geq0,$ we have
\begin{equation}\label{nn}
N_{k+1}\leq N_{k}+\widetilde{C}\frac{M_{k}}{\lambda^{k}}\leq\widetilde{C}P_{k},\quad |N_{k+1}|\leq |N_{k}|+\widetilde{C}\frac{M_{k}}{\lambda^{k}}\leq\widetilde{C}P_{k},
\end{equation}
$$\xi_{k}\leq4C_{2}\lambda^{2}\leq \frac{1}{4}\lambda,
$$
\begin{equation}\label{eta}
\eta_{k}\leq C_{1}\lambda^{k}\sigma(\lambda^{k})+C_{1}|N_{k}|\lambda^{k}\omega(\lambda^{k})
+C_{3}\lambda^{k}\left(\varphi(\|u\|_{L^{\infty}(\Omega_{\lambda^{k}})})+(T+|N_{k}|)(\omega_{1}(\lambda^{k})+\omega_{2}(\lambda^{k}))\right),
\end{equation}
where Definition $\ref{boundarydef}$, $\ref{bvdini}$, Assumption $\ref{as1}$ and $(\ref{omega})$, $(\ref{lambda})$ are used. Then the iteration result $(\ref{induction})$ implies that
$$\frac{M_{k+1}}{\lambda^{k+1}}\leq\frac{\xi_{k}M_{k}}{\lambda^{k+1}}+\frac{\eta_{k}}{\lambda^{k+1}}\leq\frac{1}{4}\frac{M_{k}}{\lambda^{k}}
+\frac{\eta_{k}}{\lambda^{k+1}}.
$$
Now we estimate $\frac{\eta_{k}}{\lambda^{k+1}}$. By $(\ref{eta})$, we have
\begin{equation}\label{m2}
\frac{\eta_{k}}{\lambda^{k+1}}\leq \frac{C_{1}}{\lambda}\left(\sigma(\lambda^{k})+|N_{k}|\omega(\lambda^{k})\right)
+\frac{C_{3}}{\lambda}\left(\varphi(\|u\|_{L^{\infty}(\Omega_{\lambda^{k}})})+(T+|N_{k}|)(\omega_{1}(\lambda^{k})+\omega_{2}(\lambda^{k}))\right).
\end{equation}
Recalling the property of the modulus of continuity (see $(\ref{diniproperty})$) we have
\begin{eqnarray*}
\varphi(\|u\|_{L^{\infty}(\Omega_{\lambda^{k}})})&\leq&\varphi\left(\|u-v-N_{k}x_{n}\|_{L^{\infty}(\Omega_{\lambda^{k}})}+
\|v\|_{L^{\infty}(\Omega_{\lambda^{k}})}+\|N_{k}x_{n}\|_{L^{\infty}(\Omega_{\lambda^{k}})}\right)\\
&\leq&\varphi(M_{k}+T\lambda^{k}+|N_{k}|\lambda^{k})\\
&\leq&2(\frac{M_{k}}{\lambda^{k}}+T+|N_{k}|)\varphi(\lambda^{k}).
\end{eqnarray*}
By substituting the above inequality and $(\ref{nn})$ into $(\ref{m2})$, we obtain for $k\geq1$,
\begin{eqnarray*}
\frac{\eta_{k}}{\lambda^{k+1}}&\leq&\frac{C_{1}}{\lambda}\sigma(\lambda^{k})
+\frac{\widetilde{C}C_{1}}{\lambda}\omega(\lambda^{k})P_{k-1}
+\frac{C_{3}}{\lambda}\left(2(\frac{M_{k}}{\lambda^{k}}+T+|N_{k}|)\varphi(\lambda^{k})+(T+|N_{k}|)(\omega_{1}(\lambda^{k})+\omega_{2}(\lambda^{k}))\right)\\
&\leq&\frac{C_{1}}{\lambda}\sigma(\lambda^{k})+\frac{\widetilde{C}C_{1}}{\lambda}\omega(\lambda^{k})P_{k}
+\left(\frac{2C_{3}(\widetilde{C}+1)}{\lambda}P_{k}
+\frac{2TC_{3}}{\lambda}\right)(\varphi(\lambda^{k})+\omega_{1}(\lambda^{k})+\omega_{2}(\lambda^{k})).
\end{eqnarray*}

Then we take $k_{0}$ large enough (then fixed) such that
$$
\sum_{i=k_{0}}^{\infty} \frac{\widetilde{C}C_{1}}{\lambda}\omega(\lambda^{i}) \leq \frac{\widetilde{C}C_{1}}{\lambda\ln \frac{1}{\lambda}}  \int_{0}^{\lambda^{k_{0}-1}} \frac{\omega(r)}{r} d r \leq \frac{1}{4},
$$
$$
\sum_{i=k_{0}}^{\infty} \frac{2C_{3}\widetilde{C}}{\lambda}(\varphi(\lambda^{i})+\omega_{1}(\lambda^{i})+\omega_{2}(\lambda^{i}))\leq \frac{2C_{3}\widetilde{C}}{\lambda\ln \frac{1}{\lambda}}  \int_{0}^{\lambda^{k_{0}-1}} \frac{\varphi(r)+\omega_{1}(r)+\omega_{2}(r)}{r} d r \leq \frac{1}{4}.
$$
For such $k_{0}(\geq 1),$ we have
$$
\sum_{i=k_{0}}^{\infty}\sigma(\lambda^{i}) \leq \frac{1}{\ln \frac{1}{\lambda}} \int_{0}^{1} \frac{\sigma(r)}{r} d r \leq \frac{1}{\ln \frac{1}{\lambda}},
$$
$$\sum_{i=k_{0}}^{\infty}(\varphi(\lambda^{i})+\omega_{1}(\lambda^{i})+\omega_{2}(\lambda^{i}))
\leq \frac{1}{\ln \frac{1}{\lambda}} \int_{0}^{1} \frac{\varphi(r)+\omega_{1}(r)+\omega_{2}(r)}{r} d r\leq \frac{3}{\ln \frac{1}{\lambda}}.
$$
Therefore for each $k \geq k_{0},$ we have
\begin{eqnarray*}
\sum_{i=k_{0}}^{k}\frac{\eta_{k}}{\lambda^{k+1}}&\leq&\sum_{i=k_{0}}^{k}\frac{C_{1}}{\lambda}\sigma(\lambda^{i})
+\sum_{i=k_{0}}^{k}\frac{\widetilde{C}C_{1}}{\lambda}\omega(\lambda^{i})P_{i}+\sum_{i=k_{0}}^{k}\left(\frac{2C_{3}(\widetilde{C}+1)}{\lambda}P_{i}+\frac{2TC_{3}}{\lambda}\right)(\varphi(\lambda^{i})+\omega_{1}(\lambda^{i})+\omega_{2}(\lambda^{i}))\\
&\leq&\frac{C_{1}}{\lambda\ln \frac{1}{\lambda}}+P_{k+1}\left(\sum_{i=k_{0}}^{k}\frac{\widetilde{C}C_{1}}{\lambda}\omega(\lambda^{i})+
\sum_{i=k_{0}}^{k}\frac{2C_{3}(\widetilde{C}+1)}{\lambda}(\varphi(\lambda^{i})+\omega_{1}(\lambda^{i})+\omega_{2}(\lambda^{i}))\right)+\frac{6TC_{3}}{\lambda\ln \frac{1}{\lambda}} \\
&\leq& \frac{C_{1}+6TC_{3}}{\lambda\ln \frac{1}{\lambda}}+\frac{1}{2} P_{k+1}.
\end{eqnarray*}
It follows that
\begin{eqnarray*}
P_{k+1}-P_{k_{0}}=\sum_{i=k_{0}}^{k} \frac{M_{i+1}}{\lambda^{i+1}}
&\leq&\frac{1}{4}\sum_{i=k_{0}}^{k} \frac{M_{i}}{\lambda^{i}}
+\sum_{i=k_{0}}^{k} \frac{\eta_{i}}{\lambda^{i+1}}\\
&\leq&\frac{1}{4}P_{k+1}+\frac{C_{1}+6TC_{3}}{\lambda\ln \frac{1}{\lambda}}+\frac{1}{2} P_{k+1}\\
&=&\frac{3}{4}P_{k+1}+\frac{C_{1}+6TC_{3}}{\lambda\ln \frac{1}{\lambda}}.
\end{eqnarray*}
Then for all $k \geq k_{0}$,
$$
P_{k+1} \leq \frac{4C_{1}+24TC_{3}}{\lambda\ln \frac{1}{\lambda}}+4 P_{k_{0}}.
$$
Therefore $\left\{P_{k}\right\}_{k=0}^{\infty}$ is bounded. We already proved $\sum\limits_{i=0}^{\infty} \frac{M_{i}}{\lambda^{i}}$ is convergent and $\left\{N_{i}\right\}_{i=0}^{\infty}$ is bounded.

Furthermore, by $(\ref{nn})$ and the definition of $P_{i}$ it's easy to see
$$
N_{i+1}-N_{i} \leq \widetilde{C} \frac{M_{i}}{\lambda^{i}}=\widetilde{C} P_{i}-\widetilde{C} P_{i-1}, \quad \text { for }~~ i \geq 1,
$$
and
$$
N_{i+1}-\widetilde{C}P_{i} \leq N_{i}-\widetilde{C} P_{i-1}, \quad \text { for }~~ i \geq 1.
$$
So $\left\{N_{i}-\widetilde{C}P_{i-1}\right\}_{i=1}^{\infty}$ is a bounded and non-increasing sequence and $\lim\limits_{i \rightarrow+\infty}(N_{i}-\widetilde{C}P_{i-1})$ exists. In conclusion $\lim\limits_{i \rightarrow+\infty} N_{i}$ exists and we set $\tau:=$
$\lim\limits_{i \rightarrow+\infty} N_{i} .$ The proof is finished.
\
\end{proof}
\begin{lm}\label{lim}$\lim\limits_{i\rightarrow+\infty}\frac{M_{i}}{\lambda^{i}}=0.$
\end{lm}
\begin{proof}
The proof is straightforward from Lemma \ref{con} since $\sum\limits_{i=0}^{\infty} \frac{M_{i}}{\lambda^{i}}$ is convergent.
\end{proof}
\begin{lm}For each $i=0,1,2, \ldots,~$ there exists $B_{i}$ such that $\lim\limits_{i \rightarrow \infty} B_{i}=0$ and that $$\left\|u-v-\tau x_{n}\right\|_{L^{\infty}(\Omega_{\lambda^{i}})} \leq B_{i} \lambda^{i}.$$
\end{lm}
\begin{proof}For any $i \geq 0$ we have
$$
\left\|u-v-\tau x_{n}\right\|_{L^{\infty}(\Omega_{\lambda^{i}})} \leq\|u-v-N_{i}x_{n}\|_{L^{\infty}(\Omega_{\lambda^{i}})}+\|N_{i}x_{n}-\tau x_{n}\|_{L^{\infty}(\Omega_{\lambda^{i}})}.
$$
Using $(\ref{induction})$ we get
$$
\left\|u-v-\tau x_{n}\right\|_{L^{\infty}(\Omega_{\lambda^{i}})}  \leq M_{i}+\lambda^{i}|N_{i}-\tau|.
$$
We set $B_{i}=\frac{M_{i}}{\lambda^{i}}+|N_{i}-\tau|$, then
$$
\left\|u-v-\tau x_{n}\right\|_{L^{\infty}(\Omega_{\lambda^{i}})}  \leq B_{i}\lambda^{i}.
$$
At the same time, by Lemma \ref{con} snd \ref{lim} we have
$$
\lim _{i \rightarrow \infty} B_{i}=0.
$$
The proof is completed.
\end{proof}
\noindent {\bf Proof of Theorem  \ref{mr1}}
From above four lemmas we already show that $u-v$ is differentiable at 0. Since $v$ is a Lipschitz function, it's clear that $u$ is Lipschitz at 0.
\begin{rem}
If the domain satisfies Reifenberg $C^{1,\text{Dini}}$ condition(see Definition $\ref{exrei}$), we can also get the pointwise boundary Lipschitz regularity of $u$, that is Theorem $\ref{mr2}$. The proof is similar to above four steps, in which we need to use Lemma $\ref{rem1.4}$, i.e. the convergence of the $\vec{n}_{r}$ in different scales additionally.
\end{rem}

\end{document}